\documentclass[hidelinks,onefignum,onetabnum]{siamart220329}

\usepackage{amsmath}

\usepackage{amsfonts, psfrag, graphicx, mathtools, stackrel, geometry,amssymb,enumitem,indentfirst,float,color,enumerate,graphicx,subcaption,pifont,hyperref,mathtools,mathrsfs}
\usepackage{pbox}
\usepackage{booktabs, cellspace, hhline}
\usepackage[numbered]{mcode}

\usepackage[hyperpageref]{backref}

\usepackage{comment,multicol,multirow, xspace, diagbox}
\usepackage{makecell}
\usepackage{algorithm}
\usepackage{algpseudocode}

\numberwithin{equation}{section}
\numberwithin{figure}{section}

\usepackage{etoolbox}
\usepackage[title]{appendix}

\newtheorem{remark}{Remark}[section]

\newcommand{\commentout}[1]{{}} 

\newcommand{\bfA}{{\bf A}}

\newcommand{\bfB}{{\bf B}}

\newcommand{\bfC}{{\bf C}}

\newcommand{\bff}{{\bf f}}
\newcommand{\bfG}{{\bf G}}

\newcommand{\bfI}{{\bf I}}

\newcommand{\bfJ}{{\bf J}}

\newcommand{\bfK}{{\bf K}}

\newcommand{\bfL}{{\bf L}}

\newcommand{\bfM}{{\bf M}}

\newcommand{\bfP}{{\bf P}}
\newcommand{\bfp}{{\bf p}}

\newcommand{\bfS}{{\bf S}}

\newcommand{\bfu}{{\bf u}}

\newcommand{\bfv}{{\bf v}}

\newcommand{\IV}{\mathcal I_{\mathcal V}}
\newcommand{\IQ}{\mathcal I_{\mathcal Q}}
\newcommand{\mV}{\mathcal V}
\newcommand{\mQ}{\mathcal Q}
\newcommand{\IVk}{\mathcal I_{\mathcal{V}_k}}

\newcommand{\IQk}{\mathcal I_{\mathcal{Q}_k}}
\newcommand{\IQkp}{\mathcal{I}_{\mathcal{Q}_{k+1}}}
\newcommand{\IVki}{\mathcal I^{-1}_{\mathcal{V}_k}}

\newcommand{\IQki}{\mathcal I^{-1}_{\mathcal{Q}_k}}
\newcommand{\IQkpi}{\mathcal I^{-1}_{\mathcal{Q}_{k+1}}}

\usepackage{comment,multicol,xspace}
\usepackage[all]{xy} 
\usepackage{tikz-cd}
\usepackage{adjustbox}
\usepackage{multirow}

\makeatletter
\MHInternalSyntaxOn
\def\MT_leftarrow_fill:{%
  \arrowfill@\leftarrow\relbar\relbar}
\def\MT_rightarrow_fill:{%
  \arrowfill@\relbar\relbar\rightarrow}
\newcommand{\xrightleftarrows}[2][]{\mathrel{%
  \raise.55ex\hbox{%
    $\ext@arrow 0359\MT_rightarrow_fill:{\phantom{#1}}{#2}$}%
  \setbox0=\hbox{%
    $\ext@arrow 3095\MT_leftarrow_fill:{#1}{\phantom{#2}}$}%
  \kern-\wd0 \lower.55ex\box0}}
\MHInternalSyntaxOff
\makeatother

\newcommand{\grad}{\operatorname{grad}}
\newcommand{\curl}{\operatorname{curl}}
\renewcommand{\div}{\operatorname{div}}

\newcommand{\dd}{\,{\rm d}}

\newcommand{\inprd}[1]{\langle#1\rangle}

\newcommand{\vertiii}[1]{{\left\vert\kern-0.25ex\left\vert\kern-0.25ex\left\vert #1
    \right\vert\kern-0.25ex\right\vert\kern-0.25ex\right\vert}}
    \newcommand{\vertii}[1]{{\left\vert\kern-0.25ex\left\vert #1
    \right\vert\kern-0.25ex\right\vert}}

\makeatother

\usepackage{xr-hyper}
\usepackage{hyperref}

\begin{document}

\headers{Transformed primal-dual methods with variable-preconditioners}{L. Chen, R. Guo and J. Wei}
\title{
Transformed primal-dual methods with variable-preconditioners
\thanks{Submitted to the editors
DATE.
\funding{The work of the first and third authors was funded by NSF DMS-2012465. 
The work of the second author was supported in parts by NSF DMS-2309777 and PMA Funding Support (CUHK).}}}
\author{
Long Chen \thanks{Department of Mathematics, University of California, Irvine, CA 92697, USA. 
 (\email{chenlong@math.uci.edu}).}
\and Ruchi Guo \thanks{Department of Mathematics, The Chinese University of Hong Kong, Hong Kong (ruchiguo@cuhk.edu.hk)}
\and Jingrong Wei \thanks{Department of Mathematics, University of California, Irvine, CA 92697, USA. (\email{jingronw@uci.edu}).}
}

\date{}

\maketitle
\begin{abstract}
This paper introduces a novel Transformed Primal-Dual with variable-metric/preconditioner (TPDv) algorithm, designed to efficiently solve affine constrained optimization problems common in nonlinear partial differential equations (PDEs). Diverging from traditional methods, TPDv iteratively updates time-evolving preconditioning operators, enhancing adaptability. The algorithm is derived and analyzed, demonstrating global linear convergence rates under mild assumptions. Numerical experiments on challenging nonlinear PDEs, including the Darcy-Forchheimer model and a nonlinear electromagnetic problem, showcase the algorithm's superiority over existing methods in terms of iteration numbers and computational efficiency. The paper concludes with a comprehensive convergence analysis.
\end{abstract}



\tableofcontents

\section{Introduction}
Let $\mV$ and $\mQ$ be two Hilbert spaces with the inner products $\inprd{\cdot}_{\mV}$ and $\inprd{\cdot}_{\mQ}$, respectively.
Given a nonlinear function $f:\mV \rightarrow \mathbb{R}$ and a linear surjective operator $B:\mV\rightarrow \mQ$, we consider the following constrained optimization problem:
\begin{equation}
\label{eq:constraintopt} 
\min_{u\in \mV} ~ f(u), \quad \text{s.t.} ~ B u = b,
\end{equation}
where $b\in \mathcal{Q}$ is given. By introducing the  
 Lagrangian multiplier $p$, \eqref{eq:constraintopt} can be written as a saddle point problem:
\begin{equation}
\label{eq:minmax}
\min_{u\in\mV} \max_{p\in\mQ} ~ \mathcal{L}(u,p) =  f(u) + \langle B u - b, p \rangle_{\mQ}.
\end{equation}
Such a constrained optimization problem is of particular importance as it widely appears in a large class of science and engineering applications. 
A large family of nonlinear partial differential equations (PDEs) can be formulated into \eqref{eq:constraintopt} where $f$ represents the underlying physical energy and $B$ is usually from certain conservation law.
Another group of applications pertains to  inverse problems, where $f$ is a regularization or data-mismatching functional and $B$ represents the forward operator. 
In this paper, we focus on the application for solving nonlinear PDEs.

The Uzawa method~\cite{1958ArrowHurwiczUzawa} remains one of the earliest and most fundamental methods for solving \eqref{eq:minmax}. 
However, the classical Uzawa method requires solving a nonlinear optimization problem at each iteration, which can be very time-consuming, especially when the problem is ill-conditioned. Subsequent efforts have been devoted to developing inexact Uzawa methods (IUMs), where the sub-problem at each iteration merely needs to be solved to certain accuracy~\cite{1997BramblePasciakVassilev,chen1998global,cheng2003inexact,2002HuZou,2007HuZou}. 
Projected gradient descent (PGD) methods are also widely used to solve the constrained optimization problems, see~\cite{2004BaoDu,1981Dunn,2018GuptaJinNguyenMcCann,2020HenningPeterseim} and the references therein.
However, efficient computation of the projection is critical for PGD which also necessitates solving certain subproblems.
Then, preconditioners play a crucial role in accelerating solving these sub-problems, and they need to be time-dependent in order to incorporate the most recent nonlinear information.

To the best of our knowledge, several critical issues with existing IUMs remain unaddressed. On one hand, although global convergence can be achieved as demonstrated in~\cite{chen1998global}, no specific convergence rate has been obtained. On the other hand, existing techniques in the literature predominantly rely on spectrum analysis. A consequence is the attainment of local convergence, implying that the initial guess must be sufficiently close to the solution, as illustrated in~\cite{2007HuZou} for instance. In a more recent study,~\cite{song2019inexact} achieved a global linear convergence rate, assuming that the problem exhibits a specific structure, and an approximated nonlinear part can be solved exactly. We refer to Table \ref{table: existALG} for a survey of existing Uzawa-type methods for nonlinear saddle point problems.

In this work, rather than working on the conventional IUM, we shall consider the Transformed Primal-Dual (TPD) algorithm proposed in~\cite{chen2023transformed} where the authors have successfully shown global linear convergence rate.
However, a major limitation in~\cite{chen2023transformed} arises as the algorithm and analysis are confined to merely non-variant preconditioning operators, which then limits the full realization of the anticipated high efficiency.
In addition, the efficiency is not robustly supported by numerical examples, especially concerning numerous nonlinear PDEs.
In this paper, our main objective is to further improve the TPD flow to incorporate time-evolving preconditioning operators corresponding to varying metrics in the resulting algorithms. 

To emphasize the difference, we name the one in the present work Transformed Primal-Dual with variable-metric/preconditioner (TPDv) algorithm, as described in Algorithm \ref{alg:intro TPDv}. The parameters, $\alpha_k$ and $\gamma_k$ are determined by the convexity and Lipschitz constants. The operator $\tilde{S}_k$ is an effective approximation to $S_k:=B\IVk^{-1}B^{\top}$, which may not have an explicit form but the preconditioner $\IQk^{-1}$ is easily computable through the update in Line \ref{alg: TPDv IQ}. We point out that Line \ref{alg: TPDv u1}-\ref{alg: TPDv u2} covers the IUM by taking $\alpha_k=1$. Remarkably, the preconditioner of the dual variable $p$ is varying according to Line \ref{alg: TPDv IQ}, which can benefit both analysis and computation and, to our best knowledge, has never been considered in the literature.

\begin{algorithm}
\caption{Transformed primal-dual method with variable-metric/preconditioner}
\label{alg:intro TPDv}
\begin{algorithmic}[1]
\Require initial guess $(u_0, p_0)$ and ${\IQ}_{0}$ being symmetric and positive definite (SPD), a sequence of SPD operators $\{ \IVki , \tilde{S}_k\}_{k\geq 0}$ and positive scalars  $\{\alpha_k, \gamma_k\}_{k\geq 0}$

\For{$k = 1, 2, \cdots$}
\State $u_{k+1/2} = u_k - \IVki (\nabla f(u_k) + B^{\top} p_{k})$ \label{alg: TPDv u1}
\State ${\IQ}_{k+1} =\displaystyle \frac{1}{1+ \alpha_k\gamma_{k}}(\IQk + \alpha_k \gamma_{k}\tilde{S}_{k} )$  \label{alg: TPDv IQ}
\State $p_{k+1}  = p_k + \alpha_k \IQkpi (B u_{k+1/2} - b )$  \label{alg: TPDv p}
\State $ u_{k+1} = (1-\alpha_k)u_k + \alpha_k u_{k+1/2}$ \label{alg: TPDv u2}
\EndFor
\end{algorithmic}
\end{algorithm}

Although evolving preconditioners significantly enhance efficiency, they present substantial challenges in terms of analysis. The widely-used spectrum analysis for the linear saddle point problems~\cite{notay2019convergence} requires the approximation to be sufficiently close to the true solution, leading to the evolving preconditioners exhibiting nearly linear behavior. But the current work deviates from this assumption.

 By leveraging the dynamics at the continuous level and employing a meticulous Lyapunov analysis, we successfully establish the global linear convergence rate of the proposed TPDv algorithms. The application of dynamic flow to assist in the design of numerical schemes for constrained optimization problems is particularly intriguing, as evidenced by works such as \cite{luo2022primal, zhuang2019efficient}. The core elements of our success hinge on the delicate design of a Lyapunov function and a flow $\IQ' =  \gamma(\tilde{S} - \IQ)$ that guides the update strategy in Algorithm \ref{alg:intro TPDv} (Line \ref{alg: TPDv IQ}). This concept is inspired by and extends our recent work on Nesterov acceleration, as presented in \cite{luo2022differential} on the Nesterov acceleration.
To the best of our knowledge, the literature has not yet documented any global linear convergence rate achieved with variable preconditioners. This highlights the novelty and significance of our theoretical contributions.

It is well-known that preconditioners should be problem-dependent, and their design plays a crucial role in achieving the desired efficiency.
Our second contribution is to apply the proposed method to some challenging nonlinear PDEs with designing the efficient and effective preconditioners.
This includes the Darcy-Forchheimer model~\cite{ph1901wasserbewegung,ruth1992derivation} arising from flow in porous-media and the nonlinear eddy current problem in ferromagnetic materials~\cite{ida2013electromagnetics}. 
With extensive numerical results, we demonstrate that the proposed TPDv algorithm outperforms many existing ones in the literature in the sense of both iteration numbers and computational time.

This article consists of additional 5 sections. In the next section, we derive the TPDv flow and algorithms. Numerical results for solving the nonlinear PDEs are reported in Sections \ref{sec:darcy} and \ref{sec:maxwell}. In Section \ref{sec:converg}, we provide a comprehensive convergence analysis at both the continuous and discrete levels. Concluding remarks are addressed in Section \ref{sec:conclusions}.

\section{Transformed Primal-Dual Methods with Variable-Preconditioners}
\label{sec:TPDv}
In this section, we briefly motivate the transformed primal-dual with variable-preconditioners (TPDv) flow and derive the optimization algorithm. We treat $(u(t), p(t))$ as time-dependent functions and $(\cdot)'$ as the derivative with respect to $t$. 

\subsection{TPDv flow}
Let $\mathcal{A}: \mV \rightarrow \mV$ be the gradient operator of $f$; namely, $\mathcal{A}u := \nabla f(u)$, provided that $f$ is differentiable. Here with the Riesz representation theorem, $\mathcal A u$ is understood as an element in $\mV$. Note that $\mathcal{A}$ is nonlinear in general.
To motivate the TPDv flow, we temporarily assume $\mathcal{A}$ is a linear and symmetric positive definite (SPD) operator.
The following matrix decomposition is clear:
\begin{equation}
\label{eq: mat_decomp}
\left[\begin{array}{cc}
\mathcal{A} & B^{\top} \\
B & 0
\end{array}\right]
= 
\left[\begin{array}{cc}
\mathcal{I} & 0 \\
B\mathcal{A}^{-1} & \mathcal{I}
\end{array}\right]
\left[\begin{array}{cc}
\mathcal{A} & B^{\top} \\
0 & S_{\mathcal A}
\end{array}\right],
\end{equation}
where $S_{\mathcal A}= B\mathcal{A}^{-1}B^{\top}$ is the Schur complement of $\mathcal A$. As $B$ is surjective and $\mathcal A$ is SPD, $S_{\mathcal A}$ is SPD. With such factorization, the primary variable $u$ and the dual variable $p$ are decoupled. 
However, it is generally very expensive to compute $\mathcal{A}^{-1}$, making $S_{\mathcal A}^{-1}$ not computable. 

We introduce an operator $\IV^{-1}$ approximating $\mathcal{A}^{-1}$ and obtain the following decomposition
\begin{equation*}
\left[\begin{array}{cc}
\mathcal{A} & B^{\top} \\
B & 0
\end{array}\right]
= 
\left[\begin{array}{cc}
\mathcal{I} & 0 \\
-B\IV^{-1} & \mathcal{I}
\end{array}\right]
\left[\begin{array}{cc}
\mathcal{A} & B^{\top} \\
B(\mathcal{I} - \IV^{-1}\mathcal{A} ) & S
\end{array}\right],
\end{equation*}
where $S = B\IV^{-1} B^{\top}$. 
Note that even computing $S$ is challenging, and this issue will be addressed later by further introducing its approximation.
In addition, we further introduce $\IQ^{-1}$ as a preconditioner of $S$. In our recent work~\cite{chen2023transformed}, we introduced the following TPD flow
\begin{equation}
\label{eq:TPD_flow}
 \left[\begin{array}{c}
u' \\
p'
\end{array}\right]
= 
\left[\begin{array}{cc}
\IV^{-1} & 0 \\
0 & \IQ^{-1}
\end{array}\right]
\left(
\left[\begin{array}{cc}
-\mathcal{A} & -B^{\top} \\
B(\mathcal{I} - \IV^{-1}\mathcal{A} ) & S
\end{array}\right]
\left[\begin{array}{c}
u \\
p
\end{array}\right] -
\left[\begin{array}{c}
0 \\
b
\end{array}\right] \right),
\end{equation}
for which the saddle point $(u^{\star},p^{\star})$ of \eqref{eq:minmax} is clearly the static state of the TPD flow \eqref{eq:TPD_flow}.

We point it out that $\IV^{-1}$ and $\IQ^{-1}$ serve as certain preconditioners from the perspective of solving linear systems. We highlight that the preconditioners generally have to be updated to accommodate the latest nonlinear information. 
Namely, $\IV=\IV(t)$ and $\IQ=\IQ(t)$ are also functions of time $t$. 
In order to construct a suitable $\IQ$, we introduce its dynamics and propose the following TPD flow with variable-preconditioners and abbreviated as TPDv:
\begin{subequations}\label{eq:TPDv flow}
\begin{align}
u' &= \IV^{-1} ( -\nabla f(u)  - B^{\top}p ), \\
p' &= \IQ^{-1} ( Bu - b - B \IV^{-1}(\nabla f(u) - B^{\top} p) ),\\
\IQ' &= \gamma(\tilde{S} - \IQ), \label{eq: IQ_flow}
\end{align}
\end{subequations}
where $\tilde S^{-1}(t)$  is an easily computable SPD preconditioner of $S(t) = B\IV^{-1}(t) B^{\top} $.
As the nonlinear $\nabla f$ is considered in this work, $\IV$ and $\IQ$ are undoubtedly time dependent for achieving successful preconditioning, which, however, adds significant difficulties in the analysis. The flow \eqref{eq: IQ_flow} for $\IQ$ is introduced to guarantee a reduction of a tailored Lyapunov function, which serves as one key novelty in the present work.
By the ordinary differential equation theory, $\IQ(t)$ is SPD for all $t\geq 0$ given $\{\gamma(t)\}_{t \geq 0}$ is positive and $\IQ(0), \{\tilde S(t)\}_{t \geq 0}$ are SPD.

\subsection{TPDv algorithms}
We proceed to discuss the algorithms induced from the TPDv flow. 
We first introduce a sequence of linear SPD operators $\{\IVk, \tilde S_k\}_{k\geq 0}$ whose choices will be problem dependent.
Start from an initial guess $(u_0, p_0, \mathcal I_{\mQ_0})$ with $\mathcal I_{\mQ_0}$ being SPD. 
By applying a straightforward discretization to \eqref{eq:TPDv flow}, we obtain 
\begin{equation}\label{eq: EE scheme}
\left\{\begin{aligned}
u_{k+1} &= u_k - \alpha_{k} \IVki (\mathcal A(u_k) + B^{\top} p_{k}), \\
 \IQkp& = \IQk + \alpha_{k}\gamma_k(\tilde{S}_k -  \IQkp), \\
p_{k+1}  &= p_k + \alpha_{k} \IQkpi ( Bu_k - b - B \IVki ( \mathcal{A}(u_k) +  B^{\top} p_k ) ), 
\end{aligned}\right.
\end{equation}
where $\alpha_{k}$ denotes the time step and $\gamma_k$ is a parameter to control the changing rate of preconditioners. 
The positive parameter $\gamma_k$ guarantees $\IQkp$ remains SPD. 
Algorithm \ref{alg:intro TPDv} is an equivalent but computationally favorable form in which $\IVki$ is computed only once.
The update of $(u_{k+1/2}, p_{k+1})$ is an iteration of the IUM and $u_{k+1}$ is obtained by weighted average of $u_k$ and $u_{k+1/2}$. 
In each iteration of the TPDv algorithm, one only needs to solve the two systems induced by $\IVk$ and $ \IQk$.

One can also consider the implicit-explicit (IMEX) scheme, say explicit in $p$ and implicit in $u$ for example, given by
\begin{subequations}\label{eq:TPDv-imex}
\begin{align}
u_{k+1/2} &=  u_k -  \IVki ( \mathcal{A} (u_{k})  + B^{\top}p_k ), \\
\IQkp & = \frac{1}{1+ \alpha_{k}\gamma_k }( \IQk + \alpha_{k}\gamma_k\tilde S_k),\\
p_{k+1}  &= p_k + \alpha_{k} 
\IQkpi ( Bu_{k+1/2} - b), \\
\label{eq:TPDv-imex u} u_{k+1} &= u_k -  \alpha_k \IVki ( \mathcal{A} (u_{k+1})  + B^{\top}p_{k+1}).
\end{align}
\end{subequations}
Noted that \eqref{eq:TPDv-imex u} is a nonlinear function of $u_{k+1}$, and equivalent to 
\begin{equation}\label{eq:proximal}
\arg\min_{u} f(u) + \frac{1}{2\alpha_k}\| u - u_k + \alpha_k\IVk^{-1}B^{\top}p_{k+1}\|_{\IVk}^2.
\end{equation}
Solving this nonlinear equation is generally highly expensive. 
However, for some cases, we may obtain the closed-form solution of $u_{k+1}$ for a simple choice of $\IVk$, see Section \ref{sec:darcy}. When $\IVk= sI$, \eqref{eq:proximal} is known as a proximal operator. With the proximal operator, we may extend TPDv for non-smooth functions which will be explored somewhere else.

In this work, we will prove that the TPDv Algorithm \ref{alg:intro TPDv} is globally and linearly convergent, i.e., there exists $\rho \in (0,1)$ such that
$$
\| u_k - u^*\|_{\IVk}^2 + \| p_k - p^*\|_{\IQk}^2 \leq C_0(u_0, p_0) \rho^k 
$$
for $(u_k, p_k)$ generated by the algorithm with arbitrary initial value $(u_0, p_0)$.
We will explicitly demonstrate how $\rho$ depends on condition numbers measured by variable metrics.
Remarkably, to our best knowledge, this result has never been achieved for variable preconditioners. The most related work is~\cite{song2019inexact} where the nonlinear function is assumed to have a decomposition as $A+\mathcal G$, where $A$ is a linear SPD operator and $\mathcal G$ is maximally monotone. Preconditioners can be employed for the linear part, and an exact solver for the associated nonlinear problem $\IV + \mathcal G$ is required which is similar to the implicit step \eqref{eq:TPDv-imex u} and usually more expensive.

\begin{table}[htp]
	\centering
	\caption{Various Uzawa-type algorithms for nonlinear saddle point systems. $\IVk$ and $\IQk$ refer to the variable metrics used in the update of $u$ and $p$ respectively.} 
	\renewcommand{\arraystretch}{1.25}
	\begin{tabular}{@{} c c c c  @{}}
	\toprule
Preconditioners   & Ref. & Rate & Region
 \smallskip \\  
 \hline

nonlinear $\IVk$ and linear $\IQk$    & \makecell{ Chen~\cite{chen1998global} } & -- & \makecell{ global }

\smallskip \\

nonlinear $\IVk$ and nonlinear $\IQk$    & 
 \makecell{Hu and Zou~\cite{2007HuZou} } & linear & \makecell{ local } 

\smallskip \\

nonlinear $\IVk = \IV + \mathcal G$ and linear $\IQ$ (fixed)     & 
 \makecell{Song et al.~\cite{song2019inexact} } & linear & \makecell{ global }

\smallskip \\

linear $\IVk$ and linear $\IQk$    &  \makecell{This work} & linear & \makecell{ global } 

\smallskip \\
\bottomrule
	\end{tabular}
 \label{table: existALG}
\end{table}

\section{Main Theoretical Results}
\label{sec:main_results}

In this section, we will present the main theoretical results regarding the global optimal convergence, but postpone their lengthy proofs to Section \ref{sec:converg}.
To describe our results, we need to introduce the Bregman divergence of a differentiable function $f$ as
\begin{equation}
\label{breg_div}
D_{f}(u,v) : = f(u)-f(v)- \langle \nabla f(v), u-v \rangle.
\end{equation}
Given an SPD operator $M$, 
we say 
$f \in \mathcal S_{\mu_{f,M}, L_{f,M}}$ if there exists $0 < \mu_{f,M} < L_{f,M} <\infty$ such that 
\begin{equation*}
\frac{\mu_{f,M}}{2}\left\|u-v\right\|^{2}_{M}\leq D_{f}(u,v) \leq \frac{L_{f,M}}{2}\left\|u-v\right\|^{2}_{M}, \quad \forall  u, v \in \mathcal V.
\end{equation*}
We can define a condition number for the nonlinear function $f$ under the metric $M$: $\kappa_{f}(M) = L_{f,M}/\mu_{f,M},$
which is certainly invariant with respect to the scaling of $f$. Here to emphasize the dependence of the metric $M$, we treat $M$ as a variable in notation $\kappa_f(M)$.

For an SPD operator $S$, we denote the minimal (maximal) eigenvalue of $S$ as $\lambda_{\min}(S)$ ($\lambda_{\max}(S)$). 
When $D, A$ are SPD operators, $D^{-1}A$ is symmetric with respect to $(\cdot, \cdot)_A$-inner product. 
Therefore, the eigenvalues of $D^{-1}A$ are real.
Following the notation of $\mu$ and $L$ above, we define $\mu_{ A, D} := \lambda_{\min}(D^{-1}A)$ and  $L_{ A, D} := \lambda_{\max}(D^{-1} A)$. When $f$ is quadratic, i.e. $A = \nabla^2 f$ is a constant SPD operator, $L_{f,M} = \lambda_{\max}(M^{-1}A)$ and $\mu_{f,M} = \lambda_{\min}(M^{-1}A)$. 
In addition, $A \preccurlyeq D$ denotes the case that $D-A$ is positive semidefinite. So we will have $\mu_{A,D} D \preccurlyeq  A\preccurlyeq L_{A,D} D$.

Once again, the preconditioners $\IV^{-1}$ and $\IQ^{-1}$ are both time-variant, causing difficulties for analysis, especially for showing the global convergence.
One key technique in our analysis to circumvent the difficulties is to construct the following tailored Lyapunov function
\begin{equation}\label{eq: Lyapunov}
    \mathcal{E}(u,p, \IQ) =D_f(u, u^{\star}) + \frac{1}{2} \|p-p^{\star}\|^2_{\IQ}.
\end{equation}
In addition, we also define the discrete Lyapunov function 
\begin{equation}\label{eq: discrete Lyapunov}
  \mathcal E_k =   \mathcal{E}(u_k,p_k, \IQk) =D_f(u_k, u^{\star}) + \frac{1}{2} \|p_k-p^{\star}\|^2_{\IQk}.
\end{equation}
In the next two theorems, we will present the exponential stability of the equilibrium point of the TPDv flow \eqref{eq:TPDv flow} and the linear convergence of TPDv Algorithm \ref{alg:intro TPDv}.

\begin{theorem}
\label{thm:expo stability}
Given time-dependent SPD operator functions $\IV(t)$ and $\tilde S(t)$, assume $f \in \mathcal S_{\mu_{f,\IV(t)}, L_{f,\IV(t)}}$ with $\mu_{f,\IV(t)} > 1/2$ for all $t>0$. 
If $(u(t), p(t), \IQ(t))$ solves the TPDv flow~\eqref{eq:TPDv flow} with the initial value $(u(0), p(0), \mathcal I_{\mQ}(0))$ such that $\mathcal I_{\mQ}(0)$ is SPD, and $\gamma(t) = \beta(t)\mu_{S(t), \tilde S(t)},$ where
$S(t)=B\IV^{-1}(t)B^{\top}$ and 
$\beta(t) = (\mu_{f,\IV(t)} - 1/2)\mu^{-1}_{f,\IV(t)} L^{-1}_{f,\IV(t)} \in (0,1/2)$, then the following strong Lyapunov property holds
\begin{equation}
\label{eq_SLP}
\begin{aligned}
 \mathcal E'(u, p, \IQ) \leq - \tilde\mu \, \mathcal{E}(u,p, \IQ),
\end{aligned}
\end{equation}
where $\tilde{\mu}(t) = \beta(t) \min \left \{ \mu_{f,\IV(t)}, \mu_{S(t), \tilde S(t)} \right\}$. 
Therefore, 
\begin{equation}
\label{eq:exp_eq01}
\mathcal{E}(u(t),p(t), \IQ(t)) \leq \exp\left \{- \int_0^t \tilde \mu (s)  \dd s \right \} \mathcal{E}(u(0),p(0), \IQ(0)), \quad t > 0.
\end{equation}
\end{theorem}

The assumption $\mu_{f, \IV}>1/2$ can be achieved by rescaling the objective function $f$ or the preconditioner $\IV$. For example, suppose we have a preconditioner $\IV^*$ with $\mu_{f,\IV^*} < 1/2$. Define 
$\IV = \mu_{f,\IV^*}\IV^*,$ then we have $\mu_{f, \IV} = 1$. So in later analysis, we will assume $\mu_{f, \IV} = 1$.

\begin{theorem}\label{thm: linear convergence for TPD}
Given a sequence of SPD operators $\{\IVk\}_{k\geq 0}$ and $\{\tilde S_k \}_{k\geq 0}$, 
assume  $f(u) \in \mathcal S_{1, L_{f,\IVk}}$. Let $(u_k, p_k, \IQk)$ be obtained by the TPDv Algorithm \ref{alg:intro TPDv} with the initial value $(u_0, p_0, \mathcal I_{\mQ_0})$ such that $\mathcal I_{\mQ_0}$ is SPD, the parameters are
$$
\beta_k =  \frac{1}{2L_{f, \IVk}}, \quad \gamma_k =\beta_k \mu_{S_k,\tilde S_k}, \quad \mu_{S_k,\tilde S_k} = \lambda_{\min}(\tilde S_k^{-1} S_k),
$$ 
and step size
\begin{equation}\label{eq:alphabound}
    \begin{aligned}
                       0 < \alpha_k  & \leq \frac{\beta_k}{2(L_{f,\IVk} + L_{S_k, \IQkp})} \min 
        \left \{\frac{  \mu_{S_k,\tilde S_k}}{L_{S_k, \IQkp}} , 1 \right \}, \quad L_{S_k, \IQkp} = \lambda_{\max}(\IQkpi S_k).
    \end{aligned}
\end{equation}
Then for the Lyapunov function defined by~\eqref{eq: discrete Lyapunov}, it holds that
$$
\mathcal{E}_{k+1} \leq \rho(\alpha_k)\mathcal{E}_k, \quad    \text{ with }  0 < \rho(\alpha_k) < 1.
$$
In particular, when the step size 
$$\alpha_k  = \frac{\beta_k}{4(L_{f,\IVk} + L_{S_k, \IQkp})} \min 
        \left \{\frac{ \mu_{S_k,\tilde S_k} }{L_{S_k, \IQkp}} , 1\right \},$$
we have the linear convergence rate
$$
\mathcal{E}_{k+1} \leq \left (1-\frac{1}{2}\min\left \{\frac{\alpha_k \gamma_k}{1+\alpha_k \gamma_k}, \alpha_k\beta_k \right \}\right)\mathcal{E}_k.
$$    
\end{theorem}

Several remarks on the choices of parameters and preconditioners are in order.
\begin{itemize}
\item The step size $\alpha_k$ cannot be too large since the explicit scheme usually requires sufficiently small-time step to ensure the stability. But too small step size $\alpha_k$ will slow down the convergence. The upper bound given in \eqref{eq:alphabound} is a theoretical and conservative bound. In practice, one can start from some large value, say $\alpha_k = 1.5$, and reduce it slightly if the scheme is not convergent. When a semi-implicit scheme is used, one can also enlarge $\alpha_k$. 

\smallskip
\item The parameter $\beta_k$ is essentially proportional to $\kappa^{-1}_f ({\IVk})$. So we should construct $\IVk$ so that $\kappa_{f}(\IVk)$ is small. Namely, $\IVk$ are preconditioners for the convex function $f$. 

\smallskip
\item 
With given $\IVk$, $S_k = B\IVk^{-1}B^{\top}$ is the corresponding Schur complement which is usually expensive to compute its inverse. $\tilde S_k$ should be a reliable approximation of $S_k$ so that, on one hand, $\kappa(\tilde S_k^{-1}S_k)$ is small, as a larger $\kappa(\tilde S_k^{-1}S_k)$ leads to smaller $\alpha_k$. On the other hand, the preconditioner $\IQkpi$ should be computationally feasible through the update using $\tilde S_k$.

\smallskip
\item The preconditioner $\IQkpi$ is ideally a preconditioner of $S_k$. But it is updated by Line 3 in Algorithm \ref{alg:intro TPDv}. The ratio $\tilde\kappa_{S_k}(\tilde S_k):= L_{S_k, \IQkp}/\mu_{S_k,\tilde S_k}$ can be still thought of as a condition number, which is scaling invariant, although the metric measuring the upper bound and lower bound are different. 

\smallskip
\item The parameter $\gamma_k$ will contribute to the reduction rate in a similar role to $\beta_k$.

\smallskip
\item With appropriate scaling, we can let $L_{S_k, \IQkp}\leq L_{f,\IVk}$. Then the step size is in the order of $\alpha_k = \mathcal O( \kappa_f^{-2}(\IVk) \tilde\kappa_{S_k}^{-1}(\tilde S_k))$ and the rate is $1- \mathcal O(\kappa_f^{-3}(\IVk) \tilde\kappa_{S_k}^{-1}(\tilde S_k)\min \{\mu_{S_k, \tilde S_k},1\})$. So choosing the preconditioner $\IVk$ is more important but preconditioning Schur complement is also necessary. 
In the examples below, we shall see that $\kappa_{S_k}^{-1}(\tilde S_k)$ can be guaranteed small by applying multigrid methods.

\end{itemize}

The most subtle part is the design of $\tilde S_k$. 
Notice that $\tilde S_k$ may not be the exact Schur complement $S_k = B\IVk^{-1}B^{\top}$, as the latter one is generally very expensive to compute.
In practice, $\tilde S_k$ should approximate $S_k$ through either approximating $\IVk^{-1}$ (the case in Section \ref{sec:maxwell}) or $\IQk^{-1}$ (the case in Section \ref{sec:darcy}). The latter case is even more involved, as the approximation is indirect.
To see this, let us denote $\IQk^*$ as the operator following the update with the exact Schur complement
\begin{equation}\label{eq:true_IQ}
    \IQkp^* = \omega_k \IQk^* +  (1-\omega_k) S_k, \quad \text{ with } \omega_k =  \frac{1}{1+ \alpha_{k}\gamma_k }. 
\end{equation}
Introducing $\IQk$ as an approximation of $\IQk^*$, 
and we directly define $\tilde S_k$ through the relation
\begin{equation}
 \IQkp = \omega_k  \IQk + (1-\omega_k) \tilde S_k.
\end{equation}
Note that $\tilde S_k$ does not enter the algorithm, and instead we use $S_k$ to update $\IQk^*$ and choose $\IQk$ to approximate $\IQk^*$. 
The following proposition shows that, when $\IQk$ is sufficiently close to $\IQk^*$, $\tilde S_k$ is a good approximation of $S_k$.

\begin{proposition}\label{prop:inexat_IQinv}
Suppose $\{ \IQk^*\}_{k\geq 0}$ and $\{ \IQk\}_{k\geq 0}$ follow the iteration 
$$
    \IQkp^* = \omega_k \IQk^* +  (1-\omega_k) S_k, \quad  \IQkp = \omega_k  \IQk + (1-\omega_k) \tilde S_k.
$$
With the weights $\omega_k\in (0,1)$.
   If there exist $\delta \in (0,1)$ so that for all $k\geq 0$ the spectral radius
    \begin{equation}\label{eq:spec_control}
\rho(  I - \IQk^{-1}\IQk^{*} ) \leq \frac{\delta}{1+\delta}, \quad  \text{ and }  \; 2\delta \frac{\omega_k}{1-\omega_k}   \IQk^*\preccurlyeq  \frac{1}{2} S_k,
    \end{equation}
    then
    $$(\frac{1}{2} -\delta)  S_k \preccurlyeq   \tilde S_k  \preccurlyeq (\frac{3}{2}+\delta)  S_k.$$
\end{proposition}
\begin{proof}
    Noted that the first inequality in \eqref{eq:spec_control} implies
    $$ (1-\delta)   \IQk^* \preccurlyeq  \IQk \preccurlyeq (1+\delta)   \IQk^*.$$
By the update of $ \IQk$ and $\IQk^*$, we have 
    \begin{equation*}
        \omega_k (1-\delta)  \IQk^* +(1-\omega_k) \tilde S_k \preccurlyeq  \IQkp \preccurlyeq (1+\delta)  \IQkp^*.
    \end{equation*}
    Rearranging terms and using the update of $ \IQk^*$, we get
    \begin{equation*}
        (1-\omega_k)  \tilde S_k  \preccurlyeq  (1+\delta)  \IQkp^* -  \omega_k(1-\delta)  \IQk^* \preccurlyeq (1+\delta) (1-\omega_k) S_k + 2\delta  \omega_k \IQk^*.
    \end{equation*}
    Dividing by $(1-\omega_k)$, we have the upper bound of $\tilde S_k$. The lower bound follows by symmetry.  
\end{proof}

In Section \ref{sec:darcy}, we introduce $\IQk$ for approximating $\IQk^*$ through certain inner iteration, and the number of steps is sufficiently large to make $\delta$ in \eqref{eq:spec_control} small enough so that \eqref{eq:spec_control} is satisfied. Then the convergence analysis can be applied. 
Of course, this brings a new parameter of the inner iteration step number whose choice is also problem dependent.

\section{Application I: Darcy--Forchheimer Model}\label{sec:darcy}
In this section, we apply the TPDv algorithms to the nonlinear Darcy--Forchheimer model~\cite{ph1901wasserbewegung,ruth1992derivation}. 
We refer readers to \cite{2008BrennerScott} for the standard notation of the Sobolev spaces and the associated norms as well as the finite element spaces.

\subsection{Problem setup}

Given a creep flow on a domain $\Omega\subseteq \mathbb{R}^2$, let $u$ be the velocity and $p$ be the pressure.
When the velocity is relatively high, the usual linear Darcy’s law is not valid; instead, the Darcy-Forchheimer model ~\cite{ph1901wasserbewegung,ruth1992derivation} has been proposed to describe a nonlinear relation
\begin{equation}
\label{eq:DF1}
\frac{ \mu K^{-1} + \beta |u| }{\rho} u + \nabla p =f, ~~~~ \text{in} ~ \Omega,
\end{equation}
where $K$ is the permeability tensor assumed to be
uniformly positive definite and bounded; $\mu$ is the viscosity coefficient, $\rho$ is the density of the fluid, and $\beta$ is the dynamic viscosity. The parameter $\beta$, known as the Forchheimer number, controls the strength of nonlinearity. $|\cdot|$  denotes the Euclidean vector norm $|u|^2 = u \cdot u$. 
\eqref{eq:DF1} is then coupled with the conservation law
\begin{equation}
\label{eq:DF2}
\div u = g, ~~~~ \text{in} ~ \Omega,
\end{equation}
and the Neumann boundary condition
$    u\cdot n = g_N$ on $\partial\Omega,$ where $n$ is the unit exterior normal vector to the boundary of the given domain. According to the divergence theorem, $g$ and $g_N$ are given functions satisfying the compatibility condition $\int_{\Omega} g(x) \mathrm{d} x=\int_{\partial \Omega} g_N(s) \mathrm{d} s.$

Following~\cite{2008GiraultWheeler}, we will use the function spaces $\mathbb{V} = L^3(\Omega)^2$ and $\mathbb{Q} = W^{1, \frac{3}{2}}(\Omega) \cap L_0^2(\Omega)$ where the zero mean value condition $L_0^2(\Omega)=\left\{v \in L^2(\Omega): \int_{\Omega} v(x) \mathrm{d} x=0\right\}$
is added because $p$ is defined up to an additive constant. Given $f \in$ $L^3(\Omega)^2, g \in L^ {\frac{6}{5}}(\Omega)$, and $g_N \in L^{\frac{3}{2}}(\partial \Omega)$, we arrive at the nonlinear saddle-point system in a variational form: find $(u,p) \in \mathbb{V}\times \mathbb{Q}$ such that
\begin{subequations}\label{eq:DFweak}
\label{eq:DFweak 1}
\begin{align}
( \sigma(u)  u,  v )_{\Omega} +  (\nabla p, v )_{\Omega} &=  (f,v)_{\Omega}, ~~~~ \forall v \in \mathbb{V},  \label{DF31}\\
( u, \nabla q)_{\Omega} & = (\tilde g,p) , ~~~~~~ \forall q \in \mathbb{Q} , \label{eq:DFweak 2}
\end{align}
\end{subequations}
where $(\cdot,\cdot)_{\Omega}$ denotes the usual $L^2$-inner product over $\Omega$, 
$(\tilde g,p)  :=   - (g,p)_{\Omega} +  (g_N,p)_{\partial\Omega}$, 
and $\sigma(u)$ denotes the coefficient in \eqref{eq:DF1}, i.e.,
\begin{equation}
\label{sigmau}
\sigma(u) = ( \mu K^{-1} + \beta |u| I)/\rho .
\end{equation}
The well-posedness of variational problem \eqref{eq:DFweak} and an alternating-directions iteration has been developed in~\cite{2008GiraultWheeler}.

\subsection{Discretization and preconditioning}
Let $\mathcal{T}_h$ be a triangulation of $\Omega$.
Let $N_t$ and $N_n$ be the number of elements and vertices in $\mathcal{T}_h$. 
Denoted by $V_h^1$ the linear Lagrange finite element space and $V_h^0$ the piecewise constant function space. Following~\cite{salas2013analysis}, we let $\mathcal V = [V^0_h]^2$ and $\mathcal Q = V^1_h \cap L_0^2(\Omega)$  to discretize the velocity $u$ and pressure $p$. 
The discrete versions of the operators in the bilinear form \eqref{eq:DFweak} are then given as
\begin{subequations}
\label{discrete_darcy1}
\begin{align}
& \mathcal{A}: \mV \longrightarrow \mV, ~ \text{with} ~~     (\mathcal{A}(u_h),v_h) :=  (\sigma(u_h)  u_h,  v_h)_{\Omega} , ~ \forall v_h\in \mV, \\
& B : \mV \longrightarrow \mQ, ~ \text{with} ~~   ( B u_h,q_h ): =  (u_h, \nabla q_h)_{\Omega},  ~ \forall q_h\in \mQ.
\end{align}
\end{subequations}

We proceed to describe our preconditioning strategies. Once chosen a basis of $\mV$ and $\mQ$, the function can be identified as a vector, i.e., $\mV \cong \mathbb R^{2N_t}$ and $\mQ\cong \mathbb R^{N_n}$. 
The zero average condition on the pressure $p$ is fulfilled by subtracting the integral of $p$ at the last iteration. 
Denote $\sigma_k = \sigma(u_k)$ for each iteration $u_k$. 
Then, $( \sigma_k \cdot, \cdot)_{\Omega}$ defines a weighted $L^2$-inner product.

To facilitate the presentation, 
we write the matrix form of the corresponding operators and the vector form of functions in boldface letters. 
Below are some matrices which will be frequently used.
\begin{itemize}

\item $\bfM_0^{\sigma_k}$ denotes the diagonal mass matrix on space $\mV$ under the inner product $( \sigma_k \cdot, \cdot)_{\Omega}$.
For simplicity, $\bfM_0$ is the usual mass matrix with $\sigma_k=1$.

\smallskip

\item $\bfB^{\top} = \bfM_0\bfG$, where $\bfG$ is the matrix representation of the $\grad$ operator. 

\smallskip

\item $\bfK^{-1} = \operatorname{diag}\left \{ \left[ |T_i|^{-1} \int_{T_i} K^{-1} \dd x \right]_{i=1}^{N_t} \right \}\in \mathbb{R}^{2N_t \times 2N_t}$ denotes a block diagonal matrix.

\end{itemize}

We first consider the TPDv Algorithm \ref{alg:intro TPDv}. In this case, we set $\bfI_{\mV_k} = \bfM_0^{\sigma_k}$. Then $\bfS_k =  \bfB (\bfM_0^{\sigma_k})^{-1}\bfB^{\top}$
is the negative Laplacian with variable coefficient. Choose $\bfI_{\mQ_{0}}^* = \bfS_0$. The ideal preconditioner should be an interpolation of two Laplacian with variable coefficients, i.e., for $k\geq 0$,
$$
\bfI_{\mQ_{k+1}}^*=\frac{1}{1+ \alpha_k\gamma_{k}}(\bfI_{\mQ_k}^* + \alpha_k \gamma_{k}\bfS_{k} ),
$$
which can be efficiently computed. 
However, it is relatively expensive to exactly compute $(\bfI_{\mQ_k}^*)^{-1}$.  
We set $\bfI_{\mQ_k}^{-1}$ as $m$ V-cycles of a multigrid (MG ($\bfI_{\mQ_k}^*, m$)) for computing $(\bfI_{\mQ_k}^*)^{-1}$. By Proposition \ref{prop:inexat_IQinv}, for sufficient many V-cycles, $\tilde \bfS_k$ is spectral equivalent to $\bfS_k$. 

Next, we consider the TPDv-IMEX algorithm \eqref{eq:TPDv-imex}. We use the same preconditioners, i.e., 
$$
\bfI_{\mV_k}  = \bfM_0^{\sigma_k}, \quad \bfI_{\mQ_k}^{-1} = {\rm MG}(\bfI_{\mQ_k}^*, m).
$$ but use the implicit-explicit scheme to update $u_{k+1}$ by: 
$$\bfu_{k+1} = \bfu_k +  \alpha_k \bfI_{\mV_k}^{-1} ( -\tilde{\mathcal{A}} (\bfu_{k+1})  - \bfB^{\top}\bfp_{k+1})$$
with $\tilde{\mathcal{A}} (\bfu_{k+1}): = \rho^{-1}(\mu \bfM_0 \bfK^{-1}\bfu_k + \beta \bfM_0|\bfu_{k+1}|\bfu_{k+1})$. 
When $K$ is scalar, $\bfu_{k+1}$ admits a closed-form solution since $\bfM_0^{\sigma_k}$ is scalar on each component:
$$
\bfu^i_{k+1}=\frac{1}{\eta^i_{k+1}}\bfv^i_{k+1}, \quad \forall T_i \in \mathcal T_h, ~~~ i=1,2,...,N_t,
$$
where ${\bf v}^i$ denotes the component of ${\bf v}$ on $T_i$ and
$$
\begin{aligned}
\bfv^i_{k+1} & =\frac{\sigma_k^i}{\alpha_k} \bfu^i_k-\frac{\mu}{\rho} K_{T_i}^{-1}  \bfu^i_k-(\nabla \bfp_k)^i+\bff^i, \quad \eta^i_{k+1}  =\frac{\sigma_k^i}{2 \alpha_k}+\frac{1}{2} \sqrt{\left(\frac{\sigma_k^i}{\alpha_k}\right)^2+4 \frac{\beta}{\rho}\left|\bfv^i_{k+1}\right|} .
\end{aligned}
$$
When $K$ is a general tensor, we can use piecewise scaled identities as the preconditioner in the TPDV-IMEX algorithm so that a similar close form can be used.

It is worthwhile to mention that the computation scheme~\cite{2008GiraultWheeler} can be regarded as $\bfI_{\mV_k}  = \bfM_0$. In numerical experiments, it is more efficient to incorporate the current information in the preconditioner as done in the present case.

\subsection{Numerical results}
We borrow the test problem from~\cite{lopez2009comparison}. We choose $\mu=1, \rho=1, \beta = 30, K=I$, and $\Omega \subset \mathbb{R}^2$ as the square $(-1,1)^2$. 

All of our experiments are implemented based on the software package $i$FEM~\cite{chen2009integrated}. The algorithms terminate until the infinite norm of residual decreases by a tolerance factor $10^{-6}$. For the TPDv algorithm, we set the fixed $\gamma = 1.4$ and step size $\alpha = 0.7$. For the TPDv-IMEX algorithm, we set the fixed  $\gamma = 0.9$ and step size $\alpha = 1.5$ as the implicit scheme leads to a larger step size.

\begin{figure}[htp]
  \begin{minipage}[t]{0.55\linewidth}

	\centering
	\caption{Iteration number and CPU time.}
	\renewcommand{\arraystretch}{1.25}
\resizebox{7.5cm}{!}{	\begin{tabular}{@{} c c  c c c @{}}
	\toprule
        &   $h$       &   DoFs    &   TPDv &    TPDv-IMEX
 \smallskip \\  
 \hline 
 \multirow{4}{*}{\makecell{ Iteration \\ (Vcycle)}}     &  1/64    &  49409   &   49   & 27 \\
                              &  1/128   &    197121    & 49    & 27   \\  
                               &  1/256     &    787457      &  49  &   27 \\ 
                               & 1/512     &   3147777    &   48    &  27 \\ \hline
\multirow{4}{*}{\makecell{CPU time \\(seconds)} }   &  1/64    &  49409   &   1.4   &  0.87\\
                           &  1/128   &    197121   &    6.1   &   3.7  \\ 
                             &  1/256     &    787457   &   27     &    16  \\  
                            & 1/512     &   3147777     &  122   &    73   \\   \bottomrule
\end{tabular} }
\label{table:DFmodel results}
  \end{minipage}
  \begin{minipage}[t]{0.44\linewidth}
    \centering
     \caption{Rate of CPU time growth}
    \includegraphics[scale=0.35]{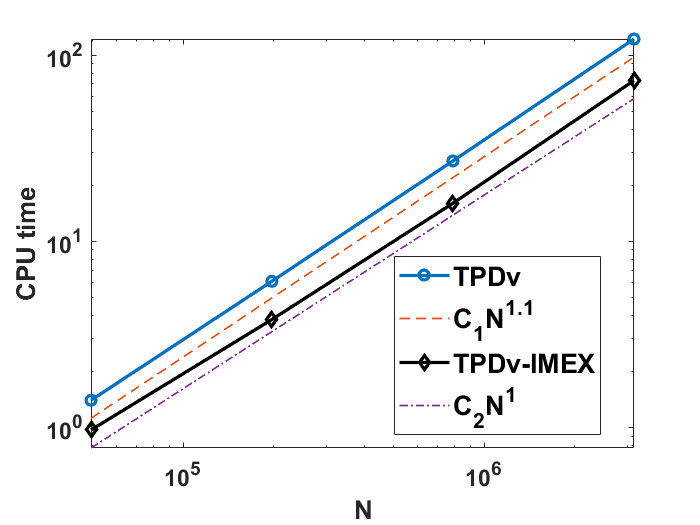}
    \label{fig:Rate of CPU time growth}
\end{minipage}
	\caption{Computation results for Darcy–Forchheimer model. For the TPDv Algorithm \ref{alg:intro TPDv}, $\gamma = 1.4$ and step size $\alpha = 0.7$. For the TPDv-IMEX algorithm \eqref{eq:TPDv-imex}, $\gamma = 0.9$ and step size $\alpha = 1.5$.}
 \label{fig:DFmodel results}
\end{figure}

The computation of $\bfI_{\mV_k}^{-1}$ is neglectable as the mass matrix of piecewise constant vectors is diagonal. The cost will be dominated by $\bfI_{\mQ_k}^{-1}$, the computation of V-cycles for Laplacian with variable coefficients. We will use one V-cycle, i.e. ${\rm MG}(\bfI_{\mQ_k}^*, 1)$. The iteration step, CPU time, and computation cost measured by V-cycle units are listed in Figure \ref{table:DFmodel results}. We can see that the iteration number is independent of grid size $h$. Due to the implicit solve for the nonlinear part, TPDv-IMEX save 1/2 computation cost. We plot the rate of CPU time growth in Figure \ref{fig:Rate of CPU time growth}. We can conclude the time growth for TPDv algorithms is almost linear, which is a near optimal complexity solver. 

In comparison to the nonlinear multigrid method used in~\cite{huang2018multigrid}, our algorithm is much simpler. By counting the number of V-cycles, our solver is significantly more efficient. In~\cite{huang2018multigrid}, when prolongate the correction from the coarse grid to the fine grid, it requires a projection to $\ker(B)$ which is equivalent to more or less $10$ V-cycles for the standard Laplacian. 

Given that nonlinear multigrid methods have demonstrated superiority over other iterative methods in~\cite{huang2018multigrid}, such as the nonlinear iteration in~\cite{2008GiraultWheeler}, we anticipate that TPDv algorithms are very efficient solvers for discretization of Darcy--Forchheimer model.

\section{Application II: A Nonlinear Electromagnetic Model}
\label{sec:maxwell}
In this subsection, we consider the nonlinear eddy current problem \cite{ida2013electromagnetics,2020XuYouseptZou,2013Yousept}. 
In the context of magnetism, such a nonlinearity exists for ferromagnetic materials.
We also refer readers to \cite{2008BrennerScott,2016Chen} for the standard notation of the Sobolev spaces and the associated norms as well as the involved finite element spaces in this section.

\subsection{Problem setup}
The nonlinear magnetostatic problem on a domain $\Omega\subseteq\mathbb{R}^3$ is: 
\begin{equation}
\label{maxwell1}
\curl \mathcal{H} = \mathcal{J}, ~~~~~~~\mathcal{H} = \nu(x,|\mathcal{B}(x)|) \mathcal{B}, ~~~~~  \div\mathcal{B} = 0, ~~~~  \text{in} ~ \Omega, 
\end{equation}
subject to suitable boundary boundary conditions, say the perfect-electric/Magnetic-conductor boundary conditions for example,
where the three-dimensional vector fields $\mathcal{H}$, $\mathcal{B}$, and $\mathcal{J}$ denote the magnetic field, the magnetic induction, and the current density, respectively.
Here, $\nu(\cdot)$ is a nonlinear function describing the magnetic reductivity which is also the inverse of the magnetic permeability.

In view of Gauss’ law for magnetism, there is a vector $u\in H_0(\curl;\Omega)\cap H(\div;\Omega)$ called magnetic potential satisfying $\mathcal{B} = \curl u$ and $\div u = 0$ in $\Omega.$
Following the standard strategy, we also introduce a Lagrange multiplier $p$ to impose the constraint $\div u=0$ in a weak sense. 
Then we arrive at the nonlinear curl-curl problem \cite{bachinger2005numerical,xu2020adaptive}: find $u\in H_0(\curl;\Omega), p\in H^1_0(\Omega)$ such that
\begin{subequations}
\label{maxwell3}
\begin{align}
 ( \nu(x,|\curl u(x)|) \curl u, \curl v )_{\Omega} + (\nabla p, v)_{\Omega}&= (\mathcal{J}, v)_{\Omega}, ~~~~ \forall v \in H_0(\curl;\Omega),  \label{maxwell31}\\
 (u, \nabla q)_{\Omega} & = 0, ~~~~~~~~~~~~\forall q \in H^1_0(\Omega). \label{maxwell32}
\end{align}
\end{subequations}
For the existence and uniqueness of the solution to the resulting system \eqref{maxwell3}, 
we need some monotonicity assumption on $\nu$ and refer readers to Lemma 3.1 in \cite{2013Yousept} for details.

\subsection{Discretization and preconditioning}
Let $\mathcal{T}_h$ be a tetrahedral mesh of $\Omega$.
For discretizing the Maxwell system above, we let  $\mathcal{V} = V^e_h(\Omega)\subset H_0(\curl; \Omega)$ and $\mathcal{Q}= V^n_h(\Omega)\subset H_0^1(\Omega)$ be the lowest order Nedelec edge element and the linear Lagrange element \cite{2008BrennerScott,2016Chen} with their associated zero boundary conditions, respectively. The dimensions of $V^n_h(\Omega)$ and $V^e_h(\Omega)$ are denoted as $N_n$ and $N_e$.
Then, the discretized operators are given as 
\begin{subequations}
\label{discrete_maxwell1}
\begin{align}
 B & := - \div_h : \mV \longrightarrow \mQ,   \text{ with }  (  B u_h,q_h ) : = (u_h, \nabla q_h)_{\Omega},  ~ \forall q_h\in \mQ, \nonumber \\
B^{\top} & := \grad : \mQ \longrightarrow \mV,  \text{ with }  (  B_h^{\top}p_h, v_h ): = (\nabla p_h, v_h)_{\Omega},  ~ \forall v_h\in \mV.
\end{align}
\end{subequations}

As $(\curl\cdot,\curl\cdot)$ is not coercive in $H_0(\curl;\Omega)$, we shall consider the augmented Lagrange method (ALM) by adding a div-div term as shown below. 
In addition, we define
\begin{align}
\label{Anu2}
 \mathcal{A}^{\nu}: \mV \rightarrow \mV, ~ \text{with} ~ (\mathcal{A}^{\nu} u_h, v_h) = (\nu \curl u_h, \curl v_h) + (\nu \div_h u_h, \div_h v_h), ~ \forall u_h, v_h\in \mV
\end{align}
which is a weighted Hodge Laplacian, crucial for designing an efficient preconditioner.

Similar to the presentation in Section \ref{sec:darcy}, by chosing a basis of $\mV$ and $\mQ$, 
all the functions in these two Hilbert spaces can be identified as vectors. 
Then, we establish the matrix representation of the related operators. 

\begin{itemize}
\item Let $\bfG:\mathbb{R}^{N_n} \rightarrow \mathbb{R}^{N_e}$ be the grad matrix and let $\bfC:\mathbb{R}^{N_e} \rightarrow \mathbb{R}^{N_f}$ be the curl matrix. Those matrices are indeed the incidence matrices between sub-simplexes if correct scaled bases are chosen.

\smallskip
\item  With a fixed $\nu$, define $\bfM^{\nu}_{\dd}\in \mathbb{R}^{N_{\dd}\times N_{\dd}}$ as the mass matrices of the Lagrange, N\'ed\'elec and Raviart-Thomas shape functions with a weight function $\nu$, respectively, for $\dd = n,e,f$. In particular, if $\nu=1$, $\bfM^{1}_{\dd}$ denotes the standard mass matrix.

\smallskip
\item The matrix associated with $(Bu_h, q_h)_{\Omega}:=(\div_h u_h,q_h)_{\Omega}$ is $\bfB = \bfG^{\top} \bfM_e$.

\smallskip
\item With a fixed $\nu$, the matrix associated with $(\mathcal{A}^{\nu} \cdot , \cdot)_{\Omega}$ is 
$$
\bfA^{\nu} = \bfC^{\top} \bfM^{\nu}_f \bfC + \bfB^{\top} \bfM^{-1}_n \bfM^{\nu}_n  \bfM^{-1}_n \bfB.
$$ 

\item The matrix associated with $(\mathcal{L} u_h, v_h)_{\Omega} := (\nabla u_h, \nabla v_h)_{\Omega}$ is $\bfL =\bfG^{\top}\bfM_e\bfG$.
\end{itemize}

A good preconditioner would be $(\bfA^{\nu})^{-1}$, which is quite expensive. 
As the Hodge Laplacian can be solved very efficiently by multigrid(MG) methods \cite{2018ChenWuZhongZhou}, 
we choose
$$
\bfI_{\mV}^{-1} = {\rm MG}(\bfA^{\nu};\epsilon_{\rm mg})
$$
--the associated MG solver--as the approximation to $(\bfA^{\nu})^{-1}$, 
where $\epsilon_{\rm mg}$ controls the number of Vcycles by reducing the residual error to $\epsilon_{\rm mg}$ multiple of the initial one in each inner iteration. 
Note that $\bfI_{\mV}$ is a SPD matrix.
The tolerance $\epsilon_{\rm mg}$ does not have to be very small to achieve the fast convergence. Some specific choices of $\epsilon_{\rm mg}$ will be given for the numerical examples below.

At the next stage, we focus on developing the preconditioner for the space $\mQ$ by employing the flow  of $\IQ$ in \eqref{eq: EE scheme} and constructing a $\tilde{S}_k$ that can be efficiently computed,
since the Schur complement ${\bf S} = \bfB \bfI_{\mV}^{-1} \bfB^{\top}$ and its inverse are both expensive to compute. Here, we will use
\begin{equation}
\label{tildS}
\tilde{\bfS} = \bfB (\bfA^{\nu})^{-1} \bfB^{\top},
\end{equation}
and proceed to derive a closed-form expression for $\tilde{\bfS}$ which is essentially a mass matrix

\begin{lemma}
\label{lap_identity_disct}
There holds that
\begin{equation}
\label{lap_identity_disct_eq02}
\bfB(\bfA^{\nu})^{-1}\bfB^{\top} = \bfM_n ( \bfM^{\nu}_n)^{-1}  \bfM_n := \widetilde{\bfM}^{\nu}_n.
\end{equation}
\end{lemma}
\begin{proof}
From the matrix representation of differential operators, we need to show
\begin{equation}
\label{lap_identity_disct_eq1}
\bfM^{-1}_e ( \bfC^{\top} \bfM^{\nu}_f \bfC + \bfB^{\top} \bfM^{-1}_n \bfM^{\nu}_n  \bfM^{-1}_n \bfB ) \bfG = \bfG \bfM^{-1}_n \bfM^{\nu}_n \bfM^{-1}_n G^{\top} \bfM_e \bfG
\end{equation}
Notice that $\bfC \bfG = 0$ by the fact $\curl \grad = 0$ and  $\bfB = \bfG^{\top} \bfM_e$, which yields \eqref{lap_identity_disct_eq1}.
As for \eqref{lap_identity_disct_eq02}, we use \eqref{lap_identity_disct_eq1} to conclude $\bfA^{\nu} \bfG = \bfB^{\top} \bfM^{-1}_n \bfM^{\nu}_n \bfM^{-1}_n \bfL$ which further shows
\begin{equation}
\label{lap_identity_disct_eq2}
\bfB \bfG \bfL^{-1} \bfM_n (\bfM^{\nu}_n)^{-1}  \bfM_n  = \bfB (\bfA^{\nu})^{-1} \bfB^{\top}  .
\end{equation} 
We obtain \eqref{lap_identity_disct_eq02} from \eqref{lap_identity_disct_eq2} by $\bfB \bfG = \bfG^{\top} \bfM_e \bfG = \bfL$.
\end{proof}

In the discussion below, during iteration, $\nu$ will be updated by $\nu_k:=\nu(x,|\curl u_k(x)|)$, and at each step we update $\bfI_{Q_k}$ by
$$
\bfI_{Q_{k+1}} = \frac{1}{1+ \alpha_k\gamma_{k}}(\bfI_{Q_k} + \alpha_k \gamma_{k} \widetilde{\bfM}^{\nu_k}_n ), ~~~\text{with} ~~ \bfI_{Q_0} = \widetilde{\bfM}^{\nu_0}_n.
$$
With the use of mass lumping, where $\widetilde{\bfM}^{\nu}_n$ and $\bfI_{Q_k}$ all become diagonal, the computation of $\bfI_{\mQ_k}^{-1}$ becomes highly efficient.

\subsection{Numerical results}
We demonstrate the efficiency and effectiveness of the proposed TPDv Algorithm \ref{alg:intro TPDv}, and compare it with some other widely-used methods including the fixed-point (FP), inexact fixed-point (IFP), and the preconditioned projected gradient descent (PGD). 
Here, these methods are all accelerated by suitable preconditioners which we will elaborate shortly.
The FP reads as: given $\bfu_0$ and $\bfp_0$, for $k\geq 0$, solve the linear saddle point system
\begin{equation}
\label{FP_mat}
\left[\begin{array}{cc}
{\bfA}^{\nu_k} & {\bfB}^{\top} \\
{\bfB} & 0
\end{array}\right]
\left[\begin{array}{c}
\bfu_{k+1} \\
\bfp_{k+1}
\end{array}\right]
=
\left[\begin{array}{c}
\bfJ \\
0
\end{array}\right].
\end{equation}
In computation, we let $\epsilon_{F}=10^{-9}$ to achieve the accurate inverse for FP and $\epsilon_{F}=10^{-3}$ to accelerate the inner iteration which results in the IFP iteration.  
In addition, it is known that the system in \eqref{FP_mat} is very ill-conditioned, and thus we use the MG method based on the Hodge Laplacian \cite{2018ChenWuZhongZhou} as the preconditioner to accelerate the inner iterations of both the FP and IFP.

Next, for the PGD algorithm, we introduce the energy functional:
\begin{equation}
\label{PGD_energy}
E(u) =  \int_{\Omega} \tilde{\nu}(x, |\curl {u}(x) |) \dd x - ( \mathcal{J}, u )_{\Omega} + \frac{1}{2}\| \sqrt{\nu} Bu\|^2_{\Omega}, ~~~ \forall u\in \mathcal{V}\cap \ker(\div_h).
\end{equation}
where $\tilde{\nu}(x,s) = \int_0^s \nu(x,t)t \dd t$.
Then, we represent $u$ by the vector $\bfu$ and express the gradient as
\begin{equation}
\label{PGD_energy_grad}
\nabla E(\bfu)  = \bfA^{\nu(\bfu)} \bfu - \bfJ.
\end{equation}
The PGD algorithm reads as: 
given $u_0$ satisfying $\div_h u_0 = 0$ (or some nonhomogeneous condition) with its vector expression $\bfu_0$, then perform
\begin{equation}
\begin{split}
\label{PGD}
\bfu_{k+1} = \bfu_k  - \alpha_k  \bfP_{\bfI_{\mathcal{V}_k},\bfB}  \bfI^{-1}_{\mathcal{V}_k} \nabla E(\bfu_k),
\end{split}
\end{equation}
where the step size $\alpha_k = 1$, and $\bfP_{\bfI_{\mV_k},\bfB}$ is the projection matrix onto $\ker(\bfB)$ with respect to the inner product $(\cdot,\cdot)_{\bfI_{\mV_k}}$ associated with the SPD matrix $\bfI_{\mV_k}$.
Here, we point out that $\bfI_{\mV_k}$ could be also interpreted as the matrix to adjust the gradient direction, similar to the one used in (quasi-)Newton algorithms. 
In the numerical experiments,  we have observed that this matrix should match the metric used to define the projection for stability; otherwise the algorithm does not converge.

In this context, the selection of $\bfI^{-1}_{\mV_k} $  significantly influences the convergence speed. Similar to the TPDv, $\bfI_{\mV_k} $  should approximate ${\bfA}^{\nu}_k$  but needs to be chosen to balance its own computation and the inversion of the Schur complement. 
Our experiments suggest that $\bfI^{-1}_{\mV_k} ={\rm MG}(\bfA^{\nu}_k;\epsilon_{\rm mg})$ gives the fastest convergence,
and here, we fix $\epsilon_{\rm mg\_PGD}=0.01$.
Then, the ideal projection is $\bfP_{\bfI_{\mV_k},\bfB}   =   ( \bfI -   \bfI^{-1}_{\mV_k} \bfB^{\top}(\bfB \bfI^{-1}_{\mV_k} \bfB^{\top})^{-1} \bfB)$.
However, as the inverse of the Schur complement $\bfB\bfI^{-1}_{\mV_k}  \bfB^{\top}$ is expensive to compute,
following the spirit above we shall use the modified one:
\begin{equation}
\label{proj_mat}
\widetilde{\bfP}_{\bfI_{\mV_k},\bfB}   =   ( \bfI -   \bfI^{-1}_{\mV_k} \bfB^{\top} (\bfB (\bfA^{\nu_k})^{-1} \bfB^{\top})^{-1} \bfB).
\end{equation}
where the Schur complement $\bfB (\bfA^{\nu_k})^{-1} \bfB^{\top}$ is substituted with $\widetilde{\bfM}^{\nu_k}_n$ as derived from Lemma \ref{lap_identity_disct}; the inverse of the latter is much more accessible. 
Here, we point out that the operator in \eqref{proj_mat} deviates from a strict projection.
Nevertheless, our numerical experiments suggest that such an approach can significantly improve the efficiency.
It is worth noting, however, that even with these improvements, it does not reach the same level of effectiveness as the proposed TPDv algorithm.
Notice that one PGD iteration requires computing $\bfI^{-1}_{\mV_k}$ twice.

Next, we proceed to present the numerical experiments. We consider the modeling domain $\Omega = [-1,1]^3$, the exact solution:
\begin{equation}
\label{Maxwell_num1}
u = \cos(\omega x_1) \cos(\omega x_2) \cos(\omega x_3)
\left[\begin{array}{c}
0 \\
0 \\
1
\end{array}\right],
\end{equation}
where $\omega=4$ is the frequency, and the nonlinear electromagnetic permeability function:
\begin{equation}
\label{Maxwell_num2}
\nu(s) = a_0 + a_1 \exp(-a_2s).
\end{equation}
In the following experiments, we shall fix $a_0 = 10$ and $a_2=1$ but vary $a_1=70$ and $73.89$ to adjust the monotonicity of the function $\nu(s)s$; namely
\begin{equation}
\label{maxwell_monot_1}
(\nu(x,s)s - \nu(x,t)t)(s-t) \ge \nu_1(s-t)^2, ~~~ \forall s,t\ge 0, ~ x \in \Omega.
\end{equation}
where the constant $\nu_1$ further corresponds to the coercivity constant of this nonlinear PDE. 
We stop the iteration when the residual error decreases by a tolerance factor $ToL = 10^{-5}$.
As the convergence with respect to the mesh size is clear, here we ignore the related results.

We have also observed that the convergence is very insensitive to $\gamma$, and thus we fix $\gamma=0.5$.
In Table \ref{table:Maxwell_TPDvpara}, for the case of coefficient configuration $a_0 = 10, a_1 = 70, a_2=1$ and $\nu_1 = 0.5265$, we evaluate the performance of the TPDv algorithm with different values of the step size $\alpha$ and the tolerance $\epsilon_{\rm mg}$ in the MG solver, in terms of iteration step, CPU time, and computation cost measured by V-cycle units.
The results indicate that setting $\alpha = 1.3, \gamma = 0.5$, and $\epsilon_{\rm mg} = 0.1$ leads to a relatively more efficient configuration; 
namely the TPDv algorithm can indeed handle relatively large step size.
Furthermore, if different step sizes can be chosen for the primal and dual variables, the TPDv algorithm can potentially achieve even greater efficiency, as shown in the last column in Table \ref{table:Maxwell_TPDvpara}.

Next, in Table \ref{table:Maxwell}, we compare the TPDv Algorithm against the aforementioned FP, IFP and PGD algorithms.
In addition to the coefficient configuration $a_0 = 10, a_1 = 70, a_2=1$ and $\nu_1 = 0.5265$, we also include the case of $a_0 = 10, a_1 = 73.89, a_2=1$ and $\nu_1 = 7.59\times 10^{-5}$ where the coercivity constant becomes extremely small, making the optimization problem even more challenging.
From the results, one can clearly observe that the TPDv algorithm is much faster than all other methods tested.

\begin{table}[htp]
	\centering
\caption{Comparison of various step sizes $\alpha$ and tolerance $\epsilon_{\rm mg}$ in MG solver used in the TPDv algorithm.  Coefficient configuration: $a_0 = 10, a_1 = 70, a_2=1$ and $\nu_1 = 0.5265$.}
\label{table:Maxwell_TPDvpara}
	\renewcommand{\arraystretch}{1.25}
	\begin{tabular}{@{} c c  c c c c @{}}
	\toprule
                  & DoFs                   
&    \makecell{ $\alpha=1.3$ \\ $\epsilon_{\rm mg}=0.1$ }          
&   \makecell{ $\alpha=1$ \\ $\epsilon_{\rm mg}=0.2$ }           
&    \makecell{ $\alpha=0.9$ \\ $\epsilon_{\rm mg}=0.3$ }   
&    \makecell{ $\alpha_V=1.5$, $\alpha_Q=0.3$ \\ $\epsilon_{\rm mg}=0.2$ }     
 \smallskip \\  
 \hline 
\multirow{ 3}{*}{Iteration}      & 52832                  &    38               &     56         &     74    &   37               \\  
                                    & 440512                &    47                &    66          &     96   &   48              \\  
                                    & 3596672              &     42               &     64         &     89   &   45                \\ \hline
\multirow{ 3}{*}{Vcycle}  & 52832                &    329            &       374      &   459    &   235          \\ 
                                    & 440512                 &    461           &      453       &  466      &   325           \\  
                                    & 3596672               &    478           &      468         &  435     &   304           \\ \hline
\multirow{ 3}{*}{\makecell{CPU time \\(seconds)}}   & 52832       &    18              &     23        &    23       &    14          \\  
                                    & 440512                 &    237            &     280       &   280      &   199          \\ 
                                    & 3596672                &   2524           &     3013     &   3060   &   1807             \\ \bottomrule
\end{tabular} 
\end{table}

\begin{table}[htp]
	\centering
\caption{Comparison of various algorithms. Coefficients: $a_0 = 10$ and $a_2=1$. ``$-$" means the convergence is too slow, which cannot be finished in a reasonable time.}
\label{table:Maxwell}
	\renewcommand{\arraystretch}{1.25}
	\begin{tabular}{@{} c c | c c c c || c c c c @{}}
	\toprule
 &  & \multicolumn{4}{c||}{$a_1 = 73.89, \, \nu_1 = 7.59 \times 10^{-5}$}  & \multicolumn{4}{c}{$a_1 = 70, \, \nu_1 = 0.5265$}\\ \cline{3-10}
      & DoFs       &       FP        &  IFP       &    PGD    &    TPDv   &   FP       &  IFP      &    PGD    &    TPD
\\  
 \hline 
\multirow{4}{*}{Iteration}      & 6064       &    37       &     37      &      31     &    34       &    35        &     35      &    29      &  31      \\  
                                    & 52832   &    71      &     71      &     40      &     56       &    55       &     55       &    29       &  37      \\  
                                    & 440512   &    --      &    --      &     55     &      82       &    --       &     57       &    39       &  48     \\ 
                                    & 359667    &     --     &     --     &     44        &   87       &    --       &    --      &   36      &  45        \\ \hline
\multirow{4}{*}{Vcycle}      & 6064     &    8899     &     5055      &   537    &     190     &    8032     &    4570     &    496    &    166     \\  
                           & 52832   &   22895    &     13589    &   684    &     293     &    17025     &     9919     &   791    &    235         \\  
                                    & 440       &    --       &      --   &  1264    &    470      &    --               &     12445     &    882    &    325         \\ 
                                    & 359667        &    --    &     --  &  1201    &     478      &     --               &     --          &     960     &    304         \\ \hline
\multirow{4}{*}{\makecell{ CPU time \\    (seconds)   }}        & 6064      &    32           &    19              &    5.2     &   1.7         &    29               &      17          &    4.8     &   1.9          \\  
                                    & 52832         &    1061           &    640       &   67      &   18         &    733             &    451        &    75       &  14            \\ 
                             & 440512         &    --        &     --        &   1226     &   270       &    --            &    5644      &    856     &   199          \\
                                    & 3596672        &   --           &     --        &    11707        &  2977        &   --        &   --           &    10078      &   1807          \\ \bottomrule
\end{tabular} 
\end{table}

\section{Convergence Analysis}
\label{sec:converg}

In this section, we give the global linear convergence rate for the TPDv algorithms. 
Let us start with recalling some basic notation and concepts, see details in \cite{nesterov2003introductory}.

\subsection{Background}

In the discussion below, we overload the notation $\langle u,v\rangle_M: = \langle Mu,v \rangle $ and $\|u\|_M$, $\forall u,v\in \mV$, to represent the associated inner product and norm induced by a general linear SPD operator $M: \mV \rightarrow \mV$. We drop the subscript, especially when $M= I$.

Indeed, all these convexity, Lipschitz and conditioning constants depend on the metric $M$ and, of course, the function $f$ itself.
Furthermore, we have the following bounds using gradient norm for $f \in  \mathcal S_{\mu_{f,M}, L_{f,M}}$:
\begin{equation}
\label{eq: Breg bound}
     \frac{1}{2L_{f,M}} \|\nabla f(u) - \nabla f(v)\|^2_{M^{-1}} \leq D_f(u,v) \leq   \frac{1}{2\mu_{f,M}} \|\nabla f(u) - \nabla f(v)\|^2_{M^{-1}}  , \quad \forall  u, v \in \mV.
\end{equation}

Moreover, when $f$ is strongly convex, its convex conjugate,
\begin{equation*}
f^*: \mV \rightarrow \mathbb{R}, ~~~ \text{with} ~ f^*(\xi) = \max_{u \in \mV} ( \xi, u )- f(u),
\end{equation*}
is also convex. 
For $f \in \mathcal S_{\mu_{f,M}, L_{f,M}}$, we have $f^*\in \mathcal S_{\mu_{f^*,M^{-1}}, L_{f^*,M^{-1}}}$ and 
\begin{equation}
\label{eq: conj_rela}
\mu_{f^*,M^{-1}} = 1/L_{f,M}, \quad L_{f^*,M^{-1}} = 1/\mu_{f,M}.
\end{equation}
Furthermore, for any $u \in \mV, \xi \in \mV$, there holds
\begin{equation}
\label{eq: conj_rela_eq2}
\nabla f^*(\xi)=u \iff \nabla f(u)=\xi.
\end{equation}

\subsection{Exponential stability of the TPDv flow}

For the sake of simplicity, let us introduce the following notation to represent the right hand of the TPDv flow in \eqref{eq:TPDv flow}:
\begin{equation}\label{eq:Gnote}
\begin{split}
&\mathcal{G}^u(u,p) : =   -\mathcal{A}(u)  - B^{\top}p,  \\
& \mathcal{G}^p(u,p) : = Bu - b - B \IV^{-1}\mathcal{A}(u) - B\mathcal{I}^{-1}_\mathcal{V}  B^{\top} p, \\
&  \mathcal{G}^{Q}(\IQ) := \tilde S- \IQ.
\end{split}
\end{equation}
We further denote the TPDv flow as 
\begin{equation}
\label{eq: TPDrhs}
\left[\begin{array}{c}
u \\
p \\
\IQ
\end{array}\right]'
=\mathcal{G}(u, p, \IQ)
~~~
\text{with}~~
\mathcal{G}(u, p, \IQ)
:=
\left[\begin{array}{c}
\IV^{-1}\, \mathcal{G}^u(u,p) \\
\IQ^{-1}\, \mathcal{G}^p(u,p) \\
\gamma\, \mathcal{G}^{Q}(\IQ) 
\end{array}\right].
\end{equation}

Treating the Lyapunov function $\mathcal{E}(u,p, \IQ) =D_f(u, u^{\star}) + \frac{1}{2} \|p-p^{\star}\|^2_{\IQ}$ in \eqref{eq: Lyapunov} as a function of $t$, 
we write its temporal derivative as 
\begin{equation}
\label{eq_Ed}
    \mathcal E'(u, p, \IQ)  = \inprd{  \partial_u \mathcal{E}(u,p,\IQ) , \IV^{-1} \mathcal{G}^u(u,p) } + \inprd{  \partial_p \mathcal{E}(u,p,\IQ) , \IQ^{-1} \mathcal{G}^p(u,p) } + \frac{1}{2}\|p-p^{\star}\|^2_{\IQ'}.
\end{equation}
The crucial step for the proposed convergence analysis is to verify the strong Lyapunov property \cite{chen2021unified} which yields the exponential stability.
Namely, the main objective of this subsection is to prove Theorem \ref{thm:expo stability}.
For this purpose, given any SPD linear operator $M$, we introduce the function
\begin{equation}\label{eq:e_M}
    \begin{aligned}
    e_M(\xi) &= \xi - M\nabla f^*(\xi),
    \end{aligned}
\end{equation}
which can be thought of as one gradient descent step of $f^*$ at $\xi$ under the metric induced by $M$. Define $L_{e,M^{-1}}$ as the least positive number such that
\begin{equation}\label{eq: Le_M}
    \begin{aligned}
    \|e_M(\xi_1) - e_M(\xi_2)\|^2_{M^{-1}} \leq L_{e,M^{-1}} \|\xi_1-\xi_2\|^2_{M^{-1}},\quad \forall \xi_1,\xi_2 \in \mV.  
    \end{aligned}
\end{equation}
Notice that $e_M(\xi)$ is called a contractive map if $L_{e,M^{-1}} < 1$. We shall derive a sufficient and necessary condition for $e_M(\xi)$ being contractive in the following lemma.

\begin{lemma}
\label{lem: Equivalence between Le and muf} 
Suppose $f \in \mathcal S_{\mu_{f,M}, L_{f,M}}$. 
Then, $L_{e,M^{-1}} < 1$ if and only if $\mu_{f,M} >1/ 2$. 
Furthermore, if $\mu_{f,M} > 1/2$, we have
\begin{equation}
\label{eq: LEM}
L_{e,M^{-1}} \leq 1 -  \frac{(2\mu_{f,M} - 1)}{\mu_{f,M}L_{f,M}}.
\end{equation}
\end{lemma}
\begin{proof}
On one hand, suppose $L_{e,M^{-1}} < 1$. Then, for any $\xi_1, \xi_2 \in \mV$, we can expand the squares
\begin{equation}\label{e(u1) - e(u2) square}
    \begin{aligned}
    \|e_M(\xi_1) - e_M(\xi_2)\|_{M^{-1} }^2 ={}& 
 \|\xi_1-\xi_2\|_{M^{-1}}^2+ \|\nabla f^*(\xi_1) - \nabla f^*(\xi_2)\|_{M}^2 \\
    &- 2 \langle \xi_1-\xi_2, \nabla f^*(\xi_1) - \nabla f^*(\xi_2)\rangle .
    \end{aligned}
\end{equation}
As $L_{e,M^{-1}} < 1$, we have $\|e_M(\xi_1) - e_M(\xi_2)\|_{M^{-1}}^2  < \|\xi_1 - \xi_2\|_{M^{-1}}^2$.
Then, by \eqref{e(u1) - e(u2) square}, we obtain
\begin{equation*}
    \begin{aligned}
  \|\nabla f^*(\xi_1) - \nabla f^*(\xi_2)\|_{M}^2  &< 2\langle \xi_1-\xi_2, \nabla f^*(\xi_1) - \nabla f^*(\xi_2) \rangle \\
    & \leq 2 \|\nabla f^*(\xi_1) - \nabla f^*(\xi_2)\|_{M} \|\xi_1 - \xi_2\| _{M^{-1}},
    \end{aligned}
\end{equation*}
which implies $L_{f^*,M^{-1}} < 2$, i.e., $\mu_{f,M} > 1/2$ by \eqref{eq: conj_rela}.

On the other hand, we suppose $\mu_{f,M} > 1/2$ and thus, $L_{f^*,M^{-1}} < 2$. For $\xi_1, \xi_2 \in \mV$, we need to employ the following inequality \cite[Chapter 2]{nesterov2003introductory}
$$  \|\nabla f^*(\xi_1) - \nabla f^*(\xi_2)\|_{M}^2 \leq L_{f^*,M^{-1}} \langle  \nabla f^*(\xi_1) - \nabla f^*(\xi_2), \xi_1-\xi_2 \rangle,$$
and, by \eqref{e(u1) - e(u2) square}, we obtain
\begin{equation}
\label{eq: LEM_eq1}
\|e_M(\xi_1) - e_M(\xi_2)\|_{M^{-1} }^2 < (1 - (2-L_{f^*,M^{-1}})\mu_{f^*,M^{-1}}))\|\xi_1 - \xi_2\|_{M^{-1}}^2,
\end{equation}
which implies $L_{e,M^{-1}}  < 1$. So far, we have proved the equivalence between $L_{e,M^{-1}} < 1$ and $\mu_{f,M} >1/ 2$.
At last, \eqref{eq: LEM} follows from \eqref{eq: LEM_eq1} and the relation in \eqref{eq: conj_rela}.
\end{proof}

Switching back to $f$ and $u$, 
we then conclude the following corollary.

\begin{corollary}
\label{coro_fu12}
    Suppose $f \in \mathcal S_{\mu_{f,M}, L_{f,M}}$ with $\mu_{f,M} > 1/2$, then 
\begin{equation}\label{eq: approx ineq}
    \|\nabla f(u_1) - \nabla f(u_2) - M (u_1 - u_2)\|^2_{M^{-1}} \leq L_{e,M^{-1}} \|\nabla f(u_1) - \nabla f(u_2) \|^2_{M^{-1}}, \quad \forall u_1,u_2 \in \mV.
\end{equation}
\end{corollary}
\begin{proof}
Define $\xi_1 = \nabla f(u_1)$ and $\xi_2 = \nabla f(u_2)$.
Then, by substituting the duality relation in \eqref{eq: conj_rela_eq2} into \eqref{eq: Le_M}, 
we equivalently have the desired result.
\end{proof}
In the following theorem, we state and verity the strong Lyapunov property  \cite{chen2021unified}, which plays a crucial part in the convergence analysis.

\begin{theorem}[Strong Lyapunov property]
\label{thm: Strong Lyapunov property}
Given any two SPD matrices $\IV$ and $\IQ$, 
assume $f \in \mathcal S_{\mu_{f,\IV}, L_{f,\IV}}$ with $\mu_{f,\IV} > 1/2$. Let \ $\beta =  (\mu_{f,\IV} - 1/2)\mu^{-1}_{f,\IV} L^{-1}_{f,\IV} \in (0,1/2)$ and $\gamma =\beta \mu_{S, \tilde S}$. Then for the Lyapunov function \eqref{eq: Lyapunov} and TPDv flow \eqref{eq:Gnote}, we have 
\begin{equation*}
\inprd{  \partial_u \mathcal{E} , \IV^{-1} \mathcal{G}^u } + \inprd{  \partial_p \mathcal{E} , \IQ^{-1} \mathcal{G}^p } + \frac{1}{2}\|p-p^{\star}\|^2_{\gamma \mathcal{G}^{Q}(\IQ)}
    \leq -  \frac{\beta}{2} \|\nabla f(u) - \nabla f(u^{\star})\|_{\IV}^2  -\frac{\gamma}{2}\| p-p^{\star}\|_{\IQ}^2, 
\end{equation*}
\end{theorem}
\begin{proof}
The equilibrium point $(u^{\star}, p^{\star})$ satisfies $
\mathcal{G}^u(u^{\star},p^{\star}) = 0,$ and $\mathcal{G}^p(u^{\star},p^{\star})=0$. 
Then direct computation gives
\begin{equation*}
\begin{aligned}
      &-\inprd{  \partial_u \mathcal{E} , \IV^{-1} \mathcal{G}^u } -\inprd{  \partial_p \mathcal{E} , \IQ^{-1} \mathcal{G}^p } \\
       ={}& - \langle \nabla f(u) - \nabla f(u^{\star}),  \mathcal{G}^u(u,p) -  \mathcal{G}^u(u^{\star},p^{\star})\rangle_{\IV^{-1}} - \langle p - p^{\star},  \mathcal{G}^p(u,p)-\mathcal{G}^p(u^{\star},p^{\star})\rangle \\
               ={}&  \|\nabla f(u) - \nabla f(u^{\star})\|_{\IV^{-1}}^2+ (\nabla f(u)-\nabla f(u^{\star}), B^\top(p-p^{\star}) )_{\IV^{-1}}  \\
    &  +\|p- p^{\star}\|_S^2+(B^\top(p-p^{\star}), \nabla f(u) - \IV u - (\nabla f(u^{\star}) - \IV u^{\star}) )_{\IV^{-1}}. 
\end{aligned}
\end{equation*}
which contains $2$ positive square terms and $2$ cross terms. We write 
$$
\|p- p^{\star}\|_S^2 = \| B^{\top}(p- p^{\star})\|_{\IV^{-1}}^2, \quad\text{ as } S = B\IV^{-1}B^{\top} = B\IV^{-1/2} (B\IV^{-1/2})^{\top},
$$
and use half of the square terms to bound the first cross term:
\begin{equation*}
    \begin{aligned}
 & \frac{1}{2}\|\nabla f(u) - \nabla f(u^{\star})\|_{\IV^{-1}}^2 + \langle \nabla f(u)-\nabla f(u^{\star}), \IV^{-1}B^\top(p-p^{\star}) \rangle  +\frac{1}{2} \|p- p^{\star}\|_S^2 \\
 ={}& \frac{1}{2}\|\nabla f(u) - \nabla f(u^{\star})  + B^{\top}(p - p^{\star})\|_{\IV^{-1}}^2  \geq 0,
    \end{aligned}
\end{equation*}
which is non-negative. Next, for any $\theta > 0$, the second cross term can be founded by the Young's inequality
\begin{equation*}
    \begin{aligned}
   &\langle B^{\top}(p-p^{\star}), \nabla f(u) - \IV u - (\nabla f(u^{\star}) - \IV u^{\star})\rangle_{\IV^{-1}} \\
   \leq{}&  \frac{1}{2\theta} \|\nabla f(u) - \nabla f(u^{\star}) - \IV(u - u^{\star})\|_{\IV^{-1}}^2 + \frac{\theta}{2} \|p - p^{\star}\|_S^2  \\
    \leq{}& \frac{ L_{e, \IV}}{2\theta} \|\nabla f(u) - \nabla f(u^{\star}) \|_{\IV^{-1}}^2 + \frac{\theta}{2} \|p - p^{\star}\|_S^2.
    \end{aligned}
\end{equation*}
Therefore, with the flow of $\IQ$, we have 
\begin{equation*}
    \begin{aligned}
    &-\inprd{  \partial_u \mathcal{E} , \IV^{-1} \mathcal{G}^u } -\inprd{  \partial_p \mathcal{E} , \IQ^{-1} \mathcal{G}^p } - \frac{1}{2}\|p-p^*\|^2_{\IQ'}\\
    \geq{} &  \frac{1}{2}\left (1- \frac{ L_{e, \IV}}{\theta}\right)\|\nabla f(u) - \nabla f(u^{\star})\|_{\IV^{-1}}^2  + \frac{1}{2}(1-\theta)\| p-p^{\star}\|_{S}^2  + \frac{\gamma}{2}\| p-p^{\star}\|_{\IQ-\tilde S}^2  \\
     \geq{} &  \frac{\beta}{2}\|\nabla f(u) - \nabla f(u^{\star})\|_{\IV^{-1}}^2  + \frac{1}{2}\| p-p^{\star}\|_{\beta S - \gamma \tilde S}^2  + \frac{\gamma}{2}\| p-p^{\star}\|_{\IQ}^2 \\
     \geq {} &  \frac{\beta}{2} \left (\|\nabla f(u) - \nabla f(u^{\star})\|_{\IV^{-1}}^2  + \mu_{S, \tilde S}\| p-p^{\star}\|_{\IQ}^2 \right).
    \end{aligned}
\end{equation*}
The second inequality holds as we choose $\theta = \sqrt{ L_{e, \IV}}$ and show that 
$$ \beta = (\mu_{f,\IV} - 1/2)\mu^{-1}_{f,\IV} L^{-1}_{f,\IV}= (1-   L_{e, \IV} )/2 \leq  1 -  \sqrt{ L_{e, \IV}}.$$
In the last inequality , we set $\gamma = \beta\mu_{S, \tilde S}$ such that
$$ \gamma \tilde S \preccurlyeq \beta S,$$
and drop the nonnegative term involving $\beta 
S - \gamma \tilde S$. By a sign change we finish the proof .
\end{proof}

With the preparation above, we are ready to prove one of our main results: Theorem \ref{thm:expo stability}.

\begin{proof}(Proof of Theorem \ref{thm:expo stability}.)
Using the result of Theorem \ref{thm: Strong Lyapunov property},
\begin{equation*}
    \begin{aligned}
         \mathcal E'(u, p, \IQ)  \leq  -\frac{\beta}{2} \left(\|\nabla f(u) - \nabla f(u^{\star})\|_{\IV}^2  + \mu_{S, \tilde S}\| p-p^{\star}\|_{\IQ}^2  \right).
    \end{aligned}
\end{equation*}

With the bound of Bregman divergence \eqref{eq: Breg bound}, we have 
\begin{equation*}
    \begin{aligned}
         \mathcal E'(u, p, \IQ)  & \leq   -\beta  \left(\mu_{f}D_f(u, u^{\star})  + \frac{1}{2} \mu_{S, \tilde S}\| p-p^{\star}\|_{\IQ}^2\right) \\
         & \leq -  \tilde\mu \, \mathcal{E}(u,p, \IQ)
    \end{aligned}
\end{equation*}
where $\tilde{\mu} = \beta \min \left \{  \mu_{f,\IV(t)}, \mu_{S, \tilde S} \right\}$.
Direct integration gives the desired exponential decay in \eqref{eq:exp_eq01}.
\end{proof}

\subsection{Convergence analysis for the explicit scheme} 
We focus on the explicit scheme in \eqref{eq: EE scheme}.
Denoted by
\begin{equation}\label{eq:Gdiscrete}
    \mathcal G_{k} =\begin{pmatrix}
    \IVki \, \mathcal G^u(u_{k}, p_{k}) \\
    \IQki \, \mathcal G^p(u_{k}, p_{k})      \\   
    \gamma_k \, \mathcal G^{Q}(\IQkp)  
\end{pmatrix}:=\begin{pmatrix}
    \IVki \, \mathcal G^u_k \\
    \IQki \, \mathcal G^p_k     \\   
    \gamma_k \, \mathcal G^{Q}_{k+1}
\end{pmatrix} .
\end{equation}
 We state the discrete strong Lyapunov property in the following theorem. 

\begin{lemma}[Discrete strong Lyapunov property]
\label{thm: discrete Strong Lyapunov property}
Given any two SPD matrices $\IVk$ and $\IQk$, 
assume $f \in \mathcal S_{1, L_{f,\IVk}}$. Then for the Lyapunov function \eqref{eq: discrete Lyapunov} and $(u_k, p_k, \IQk)$ generated by the explicit scheme with initial value $(u_0, p_0, \mathcal I_{\mQ_0})$ and $\mathcal I_{\mQ_0}$ being SPD with 
$$
\beta_k =  \frac{1}{2L_{f, \IVk} },  \quad \gamma_k =\beta_k \mu_{S_k,\tilde S_k},
$$ 
we have the discrete Strong Lyapunov property
\begin{equation}
\label{eq: discrete Strong Lyapunov property}
\begin{aligned}
    &-\inprd{  \partial_u \mathcal{E}(u_k,p_k, \IQk) ,\IVki \mathcal{G}_k^u } -\inprd{  \partial_p \mathcal{E}(u_k,p_k, \IQk), \IQki  \mathcal{G}^p_k } - \frac{1}{2}\|p_k-p^{\star}\|^2_{\gamma_k\mathcal{G}^Q_{k+1}} \\
  \geq{}  & \frac{\beta_k}{2}\|\nabla f(u_k) - \nabla f(u^{\star})\|_{\IVk}^2  +  \frac{\gamma_k}{2} \| p-p^{\star}\|_{\IQkp}^2.
\end{aligned}
\end{equation}
\end{lemma}
\begin{proof}
    With the proof of Theorem \ref{thm: Strong Lyapunov property}, we have
    \begin{equation*}
    \begin{aligned}
    &-\inprd{  \partial_u \mathcal{E}_k , \IVki \mathcal{G}_k^u } -\inprd{  \partial_p \mathcal{E}_k , \IQk^{-1} \mathcal{G}_k^p } - \frac{1}{2}\|p-p^*\|^2_{\gamma_k\mathcal G^Q_{k+1}}\\
    \geq{} &  \frac{1}{2}\left (1-\sqrt{L_{e,\IVk}} \right)\|\nabla f(u_k) - \nabla f(u^{\star})\|_{\IV^{-1}}^2  + \frac{1}{2}\left (1-\sqrt{L_{e,\IVk}} \right)\| p-p^{\star}\|_{S_k}^2  + \frac{\gamma_k}{2}\| p-p^{\star}\|_{\IQkp-\tilde S_k}^2  \\
     \geq {} &  \frac{\beta_k}{2} \|\nabla f(u_k) - \nabla f(u^{\star})\|_{\IVk^{-1}}^2  + \frac{\gamma_k}{2}\| p_k-p^{\star}\|_{\IQkp}^2 .
    \end{aligned}
\end{equation*}
where  $ S_{k} = B\IVki B^{\top}$, $\beta_k$ and $\gamma_k$ are defined as specified. The formulae on $\beta_k$ is simplified using the fact $\mu_{f, \IVk} = 1$.

\end{proof}

Now we are in the position to present a proof of our main result.
\begin{proof}(Proof of Theorem \ref{thm: linear convergence for TPD}.)
We write the difference of the Lyapunov function \eqref{eq: discrete Lyapunov} as 
\begin{equation}\label{eq: E diff}
    \begin{aligned}
    \mathcal{E}_{k+1} - \mathcal{E}_k ={} &  \mathcal{E}_{k+1} - \mathcal{E}(u_k, p_k, {\IQ}_{k+1})+ \mathcal{E}(u_k, p_k, \IQkp)- \mathcal{E}_k\\
    \leq{} & \langle \partial_u  \mathcal{E}(u_k, p_k, \IQkp) , u_{k+1} - u_k \rangle + \frac{L_{f, \IVk}}{2}\|u_{k+1} - u_k\|_{\IVk}^2 \\
    &+ \langle \partial_p  \mathcal{E}(u_k, p_k, \IQkp) , p_{k+1} - p_k \rangle + \frac{1}{2} \|p_{k+1}-p_{k}\|^2_{{\IQ}_{k+1}} + \frac{1}{2} \|p_{k}-p^*\|^2_{{\IQ}_{k+1} - \IQk} .    
    \end{aligned}
\end{equation}

Using the strong Lyapunov property \eqref{eq: discrete Strong Lyapunov property} at index $k$, we have
\begin{equation}\label{eq:group1}
\begin{aligned}
    & \langle \partial_u  \mathcal{E}(u_k, p_k, \IQkp) , u_{k+1} - u_k \rangle + \langle \partial_p  \mathcal{E}(u_k, p_k, \IQkp) , p_{k+1} - p_k \rangle  + \frac{1}{2} \|p_{k}-p^*\|^2_{{\IQ}_{k+1} - \IQk} \\
   = &~ \alpha_k \inprd{  \partial_u \mathcal{E}(u_k,p_k, \IQk) ,\IVki \mathcal{G}_{k}^u }  +\alpha_k\inprd{  \partial_p \mathcal{E}(u_k,p_k, \IQk) , \IQki  \mathcal{G}^p_{k} } + \frac{\alpha_k\gamma_k}{2}\|p_k-p^{\star}\|^2_{\mathcal{G}^Q_{k+1}} \\
    \leq &~ -\frac{\alpha_k\beta_{k}}{2}\|\nabla f(u_k) - \nabla f(u^{\star})\|_{\IVk}^2  - \frac{\alpha_k\gamma_k}{2}  \| p_k-p^{\star}\|_{\IQkp}^2 
\end{aligned}
\end{equation}

We expand the square terms of difference in \eqref{eq: E diff} using Lipschitz continuity: 
\begin{align*}
\|u_{k+1} - u_k\|_{{\IV}_k}^2  ={}&\alpha_k^2 \| \nabla f(u_k) -  \nabla f(u^{\star})+ B^{\top}(p_k-p^{\star})\|_{\IVk^{-1}}^2 \\
	\leq {}&2\alpha_k^2 \|\nabla f(u_k) - \nabla f(u^{\star})\|_{\IVk^{-1}}^2  + 2\alpha_k^2 \|p_k - p^{\star}\|_{S_k}^2  , \\
 	\leq {}&2\alpha_k^2 \|\nabla f(u_k) - \nabla f(u^{\star})\|_{\IVk^{-1}}^2+ 2\alpha_k^2L_{S_k, \IQkp}\|p_k - p^{\star}\|_{\IQkp}^2, \\
 \|p_{k+1} - p_k\|_{\IQkp}^2  ={}&\alpha_k^2 \|S_k(p_k - p^{\star}) +  B[\IVk^{-1} (\nabla f(u_k) -   \nabla f(u^{\star}) )- (u_k-u^{\star})]\|_{\IQkpi}^2 \\
 \leq {}&\alpha_k^2L_{S_k, \IQkp}\|S_k(p_k - p^{\star}) +  B[\IVk^{-1} (\nabla f(u_k) -   \nabla f(u^{\star})) - (u_k-u^{\star})]\|_{S_k^{-1}}^2 \\
	\leq {}&2\alpha_k^2L_{S_k, \IQkp} \big ( \|\nabla f(u_k) - \nabla f(u^{\star}) - \IVk(u_k - u^{\star})\|_{\IVk^{-1}}^2+  \|p_k - p^{\star}\|_{S_k}^2 \big )\\
 \leq{}&2\alpha_k^2L_{S_k, \IQkp} \big (\|\nabla f(u_k) - \nabla f(u^{\star}) \|_{\IVk^{-1}}^2+  L_{S_k, \IQkp}\|p_k - p^{\star}\|_{\IQkp}^2\big ).
\end{align*}
The last inequality holds since we have the bound \eqref{eq: approx ineq} with $L_{e, \IVk} \leq 1$.

Substitute back to~\eqref{eq: E diff} and use \eqref{eq:group1}, 
\begin{equation*}
\begin{aligned}
\mathcal{E}_{k+1} - \mathcal E_k  \leq&~ -\frac{\alpha_k}{2}\left(\left[\beta_k - 2\alpha_k (L_{f, \IVk} + L_{S_k, \IQkp}) \right]\|\nabla f(u_k) - \nabla f(u^{\star}) \|_{\IVk^{-1}}^2 \right. \\
& \left. - \frac{\alpha_k}{2}\left [\gamma_k -2\alpha_kL_{S_k, \IQkp} (L_{f, \IVk} +L_{S_k, \IQkp})\right]  \|p_k-p^{\star}\|_{\IQkp}^2 \right)\\
\leq&~ -\alpha_k\left(\left[\beta_k - 2\alpha_k (L_{f, \IVk} + L_{S_k, \IQkp}) \right]D_f(u_{k}, u^{\star}) \right. \\
& \left. - \frac{\alpha_k}{2(1+ \alpha_k\gamma_{k})}\left [\gamma_k-2\alpha_kL_{S_k, \IQkp}(L_{f, \IVk} +L_{S_k, \IQkp})\right] \|p_k-p^{\star}\|_{\IQk}^2 \right) .
\end{aligned}
\end{equation*}
The last inequality holds since according to the update 
\begin{equation*}
    \IQkp =\frac{1}{1+ \alpha_k\gamma_{k}}(\IQk + \alpha_k \gamma_{k}\tilde{S}_{k} ) \geq \frac{1}{1+ \alpha_k\gamma_{k}} \IQk.
\end{equation*}
Therefore, we can conclude 
$$ \mathcal{E}_{k+1} \leq (1-\alpha_k\eta_k(\alpha_k )) \mathcal E_k$$
with 
$$
\begin{aligned}
    \eta_k(\alpha_k) &=   \min   \left \{ \left [\beta_k - 2\alpha_k (L_{f,\IVk} + L_{S_k, \IQkp}) \right] \right. , \\
    &\quad \quad \quad \left. \frac{1}{1+ \alpha_k\gamma_{k}}\left [\gamma_k- 2\alpha_k L_{S_k, \IQkp}  (L_{f,\IVk} + L_{S_k, \IQkp} ) \right]\right\}.
\end{aligned}
$$
 
The upper bound of $\alpha_k$ is given by requiring $\eta_k(\alpha_k) > 0$.  
Substitude the step size 
$$\alpha_k  = \frac{\beta_k}{4(L_{f,\IVk} + L_{S_k, \IQkp})} \min 
        \left \{\frac{ \mu_{S_k, \tilde S_k }}{L_{S_k, \IQkp}} , 1\right \},$$
we have the decay rate.
\end{proof}
\begin{remark}\rm
    For variants of TPDv algorithms such as the implicit-explicit scheme \eqref{eq:TPDv-imex}, the global linear convergence rate can be derived using the strong Lyapunov property at index $k+1$ and write the scheme as a correction of the implicit Euler scheme. We refer the reader to this strategy in \cite[Theorem 4.3] {chen2023transformed}.
\end{remark}

\section{Conclusions and Future Work}\label{sec:conclusions}

In this work, we introduce an innovative Transformed Primal-Dual with variable-metric/preconditioner (TPDv) algorithm tailored to address constrained optimization problems commonly encountered in nonlinear partial differential equations. The effectiveness of the TPDv algorithm stems from the adaptability of time-evolving preconditioning operators. We establish the dynamical system at the continuous level for the corresponding algorithm and demonstrate its exponential stability. Utilizing Lyapunov analysis, we illustrate the global linear convergence rates of the algorithm under mild assumptions and underscore the impact of the choice of preconditioners on the convergence rate. Through numerical experiments on challenging nonlinear PDEs, such as the Darcy-Forchheimer model and a nonlinear electromagnetic problem, we highlight the algorithm's superiority in terms of iteration numbers and computational efficiency compared to existing methods.

As we showcase the effectiveness of the TPDv algorithms, particularly with IMEX schemes, it becomes intriguing to extend their applicability to nonsmooth problems involving proximal operators. The analytical tools employed herein would provide theoretical guarantees for global convergence rates. In the context of large-scale computations, the TPDv algorithms, in combination with parallelizable preconditioners, could serve as efficient solvers for domain decomposition problems. 

\section*{Acknowledgments}
We would like to express our gratitude to Professor Jun Zou at the Chinese University of Hong Kong for his valuable suggestions on the presentation of the algorithms. His insights have significantly enhanced the adaptability of our theoretical framework.

\bibliographystyle{siam}
\bibliography{RuchiBib.bib}

\end{document}